\documentclass[10pt]{amsart}

\usepackage{amsfonts,amssymb,verbatim,amsmath,amsthm,latexsym,textcomp,amscd}
\usepackage{latexsym,amsfonts,amssymb,epsfig,verbatim}
\usepackage{amsmath,amsthm,amssymb,latexsym,graphics,textcomp}
\usepackage{graphicx}
\usepackage{color}
\usepackage{url}
\usepackage{hyperref}
\usepackage[T1]{fontenc}

\begin{document}

\newtheorem{theorem}{Theorem}[section]
\newtheorem{prop}[theorem]{Proposition}
\newtheorem{lemma}[theorem]{Lemma}
\newtheorem{cor}[theorem]{Corollary}
\newtheorem{defn}[theorem]{Definition}
\newtheorem{conj}[theorem]{Conjecture}
\newtheorem{claim}[theorem]{Claim}
\newtheorem{example}[theorem]{Example}
\newtheorem{rem}[theorem]{Remark}
\newtheorem{rmk}[theorem]{Remark}
\newtheorem{obs}[theorem]{Observation}
\newtheorem{qn}[theorem]{Question}
\newcommand{\map}{\rightarrow}
\newcommand{\C}{\mathcal C}
\newcommand\AAA{{\mathcal A}}
\newcommand\BB{{\mathcal B}}
\newcommand\DD{{\mathcal D}}
\newcommand\EE{{\mathcal E}}
\newcommand\FF{{\mathcal F}}
\newcommand\GG{{\mathcal G}}
\newcommand\HH{{\mathcal H}}
\newcommand\I{{\stackrel{\rightarrow}{i}}}
\newcommand\J{{\stackrel{\rightarrow}{j}}}
\newcommand\K{{\stackrel{\rightarrow}{k}}}
\newcommand\LL{{\mathcal L}}
\newcommand\MM{{\mathcal M}}
\newcommand\NN{{\mathbb N}}
\newcommand\OO{{\mathcal O}}
\newcommand\PP{{\mathcal P}}
\newcommand\QQ{{\mathcal Q}}
\newcommand\RR{{\mathcal R}}
\newcommand\SSS{{\mathcal S}}
\newcommand\TT{{\mathcal T}}
\newcommand\UU{{\mathcal U}}
\newcommand\VV{{\mathcal V}}
\newcommand\WW{{\mathcal W}}
\newcommand\XX{{\mathcal X}}
\newcommand\YY{{\mathcal Y}}
\newcommand\ZZ{{\mathbb Z}}
\newcommand\hhat{\widehat}
\newcommand\vfn{\stackrel{\A}{r}(t)}
\newcommand\dervf{\frac{d\stackrel{\A}{r}}{dt}}
\newcommand\der{\frac{d}{dt}}
\newcommand\vfncomp{f(t)\I+g(t)\J+h(t)\K}
\newcommand\ds{\sqrt{f^{'}(t)^2+g^{'}(t)^2+h^{'}(t)^2}dt}
\newcommand\rvec{\stackrel{\A}{r}}
\newcommand\velo{\frac{d\stackrel{\A}{r}}{dt}}
\newcommand\speed{|\velo|}
\newcommand\velpri{\rvec \,^{'}}
\newcommand{\RED}{\textcolor{red}}
\newcommand\gma{\Gamma}
\newcommand{\omg}{\Omega/_{\sim}}
\newcommand{\pt}{(\partial T)'}
\newcommand{\sub}{(\gamma_{\sigma(n)})_{n\in \mathbb{N}}}
\newcommand{\seq}{(\gamma_n)_{n\in \mathbb{N}}}

\newcommand{\secref}[1]{Section~\ref{#1}}
\newcommand{\thmref}[1]{Theorem~\ref{#1}}
\newcommand{\lemref}[1]{Lemma~\ref{#1}}
\newcommand{\remref}[1]{Remark~\ref{#1}}
\newcommand{\propref}[1]{Proposition~\ref{#1}}
\newcommand{\corref}[1]{Corollary~\ref{#1}}
\newcommand{\defref}[1]{Definbition~(\ref{#1})}
\title[Boundaries of graphs of relatively hyperbolic groups]{Boundaries of graphs of relatively hyperbolic groups with cyclic edge groups}

\author{Ravi Tomar}
\address{Indian Institute of Science Education and Research (IISER) Mohali, India}
\email{ravitomar547@gmail.com}

\begin{abstract}
	We prove that the fundamental group of a finite graph of convergence groups with parabolic edge groups is a convergence group. Using this result, under some mild assumptions, we prove a combination theorem for a graph of convergence groups with dynamically quasi-convex edge groups (Theorem \ref{combination}). To prove these results, we use a modification of Dahmani's technique \cite{dahmni}. Then we show that the fundamental group of a graph of relatively hyperbolic groups with edge groups either parabolic or infinite cyclic is relatively hyperbolic and construct Bowditch boundary. Finally, we show that the homeomorphism type of Bowditch boundary of the fundamental group of a graph of relatively hyperbolic groups with parabolic edge groups is determined by the  homeomorphism type of the Bowditch boundaries of vertex groups (under some additional hypotheses)(Theorem \ref{homeotype}). In the last section of the paper, we give some applications and examples.
\end{abstract}
\maketitle
\section{Introduction}
In \cite{bestcom}, Bestvina and Feighn proved a combination theorem for hyperbolic groups. Motivated by this, Mj and Reeves proved a combination theorem for relatively hyperbolic groups \cite{mjreeves}. Using the characterization of relative hyperbolicity given by Yaman (see \cite{yaman}), Dahmani proved a combination theorem for acylindrical finite graph of relatively hyperbolic groups with fully quasi-convex edge groups \cite{dahmni}. 

 In Dahmani's paper \cite{dahmni}, Theorem 0.1(1) produces convergence groups only in the case of an acylindrical finite graph of relatively hyperbolic groups with fully quasi-convex edge groups (see the main Theorem of \cite{dahmni}). From first part of this theorem, if we remove the hypothesis that vertex groups are relatively hyperbolic or edge groups are fully quasi-convex, it is unclear whether the fundamental group of a graph of groups is a convergence group. This motivates the following natural question:
\begin{qn}\label{qn}
	Let $\Gamma$ be the fundamental group of a finite graph of convergence groups with dynamically quasi-convex edge groups. Under which condition(s) is $\Gamma$ a convergence group?
\end{qn}
We answer the above question in the cases: (a) when the edge groups are parabolic (b) when the edge groups are cyclic (c) when the collection of edge groups forms a dynamical malnormal family of dynamically quasi-convex subgroups in adjacent vertex groups (Definition \ref{defn}(3),(4)). We have the following result for parabolic edge groups:
\begin{theorem}\label{para}
	Let $\Gamma$ be the fundamental group of a finite graph of countable  convergence groups with parabolic edge groups. Then $\Gamma$ is a convergence group.
\end{theorem}
{\bf Note:} Proof of parts (2),(3),(4) of Theorem 0.1 in \cite{dahmni} also go through when convergence groups replace relatively hyperbolic groups, and rest of the hypotheses remain same. Therefore, with a little effort, proof of the above theorem also follows by combining parts (2),(3),(4) of Dahmani's theorem. In part (2),(3),(4) of Dahmani's theorem, the domain of edge boundary point (see Defintion \ref{domain}) is infinite, but it is of star-like form. On the other hand, we allow domains to be infinite subtrees of the Bass-Serre tree. For example, let $\Gamma =G_1\ast_PG_2$, where $G_1,G_2$ are convergence groups and $P$ is isomorphic to parabolic subgroups and not maximal parabolic in $G_1$ and $G_2$, respectively. Suppose $P_1,P_2$ are maximal parabolic containing isomorphic copy of $P$ in $G_1,G_2$ respectively. Then domain of edge parabolic point is the Bass-Serre tree of $P_1\ast_PP_2$. Thus, we give a more direct proof by generalizing the technique of Dahmani \cite{dahmni}. In fact, we explicitly construct a compact metrizable space (see Section \ref{3}) on which $\Gamma$ acts as a convergence group. Also, we use this construction for producing Bowditch boundary in Theorem \ref{rel para}.

We state a combination theorem for convergence groups that also answers Question \ref{qn}.
 \begin{theorem}\label{combination}
 	Let $\Gamma$ be the fundamental group of a finite graph of countable convergence groups such that stabilizers of the limit sets of edge groups  form a dynamically malnormal family of dynamically quasi-convex subgroups (Definition \ref{defn} (3),(4)) in adjacent vertex groups. Then $\Gamma$ is a convergence group.
 \end{theorem}
In the above theorem, by the definition of a dynamically malnormal family (Definition \ref{defn} (4)), we see that action of $\Gamma$ on its Bass-Serre tree is $2$-acylindrical. Thus $\Gamma$ is the fundamental group of an acylindrical graph of convergence groups satisfying the hypotheses of the above theorem. When vertex groups are convergence groups and edge groups are fully dynamically quasi-convex in Theorem 0.1(1) of Dahmani \cite{dahmni}, it is different from the above theorem. In Theorem 0.1(1) of Dahmani, the limit sets of edge groups in adjacent vertex groups are homeomorphic, which is not the case here. There was no identification involved inside Bowditch boundaries of vertex groups. On the other hand, to prove the above theorem, we identify the translates of the limit set of edge groups in adjacent vertex groups with points, see Section \ref{combination proof}. We have the following proposition that gives an answer to Question \ref{qn} when edge groups are infinite cyclic.
\begin{prop}\label{cyclic}
	Let $\Gamma$ be the fundamental group of a finite graph of countable convergence groups with infinite cyclic edge groups, which are dynamically malnormal in the adjacent vertex groups. Then $\Gamma$ is a convergence group.
\end{prop} 
Note that infinite cyclic subgroups of a convergence group are dynamically quasi-convex (Lemma \ref{cyclicdyna}). The above proposition is not exactly a corollary of Theorem \ref{combination} (see Note \ref{reason 1}). In the above proposition, dynamical malnormality of edge groups can be replaced by torsion-free vertex groups (Proposition \ref{torsionfreecyclic}). Floyd boundary was introduced by W. Floyd in \cite{floyd}. One is referred to \cite{karlsson},\cite{yang} for related results. We have the following corollary:
\begin{cor}
	Let $\Gamma$ be the fundamental group of a finite graph of groups where the vertex groups are finitely generated torsion-free with non-trivial Floyd boundary and edge groups are infinite cyclic. Then $\Gamma$ is a convergence group. 
\end{cor}
\begin{proof}
	  Groups with non-trivial Floyd boundary act as a convergence group on their Floyd boundary (see \cite{karlsson}). Thus, the corollary follows from Proposition \ref{torsionfreecyclic}.
\end{proof}
 In the above corollary, if we take vertex groups to be relatively hyperbolic then $\Gamma$ is relatively hyperbolic by Theorem \ref{rel para}, and hence, by \cite{gerasimovfloyd}, it has non-trivial Floyd boundary. However, if vertex groups are not relatively hyperbolic, it is unclear whether $\Gamma$ has a non-trivial Floyd boundary; see Questions \ref{qn2} and \ref{qn3}.

Now, we state a combination theorem for a graph of relatively hyperbolic groups with parabolic edge groups. 
\begin{theorem}\label{rel para}
	Let $\Gamma$ be the fundamental group of a finite graph of relatively hyperbolic groups whose edge groups are parabolic subgroups of the adjacent vertex groups. Then $\gma$ is relatively hyperbolic.  
\end{theorem}
We explain parabolic structure of $\Gamma$ after proof of the above theorem in section 6. In the above theorem, if edge group is not maximal parabolic in adjacent vertex groups, it does not satisfy the hypotheses of \cite[Theorem 0.1]{dahmni}. However, proof of the above theorem still follows from parts (2),(3),(4) of Theorem 0.1 in \cite{dahmni}. Here, we give a different proof by constructing a compact metrizable space on which $\Gamma$ acts geometrically finitely. Therefore, we have an explicit construction of Bowditch boundary of $\Gamma$. Note that relatively hyperbolicity in Theorem \ref{rel para} also follows from the work of Bigdely and Wise \cite{bigdely}. Since in our situation, parabolic edge groups can be infinitely generated, the first condition of the theorem of Mj-Reeves \cite{mjreeves} is not satisfied. Thus, relative hyperbolicity does not follow from the theorem of Mj-Reeves. The following is a combination theorem for a graph of relatively hyperbolic groups with cyclic edge groups:
\begin{theorem}\label{rel cyclic}
	Let $\Gamma$ be the fundamental group of a finite graph of  relatively hyperbolic groups with infinite cyclic edge groups. Then $\Gamma$ is relatively hyperbolic.
\end{theorem}
 To prove the above theorem, we use Theorem \ref{rel para} (see Section \ref{6} ). In particular, by extending the parabolic structure on vertex groups, we convert a graph of reltively hyperbolic groups with infinite cyclic edge groups into a graph of relative hyperbolic groups with parabolic edge groups. Thus, we can explicitly construct Bowditch boundary for the group $\Gamma$.
 Now, we state a theorem for the homeomorphism type of Bowditch boundary of the fundamental group of a graph of relatively hyperbolic groups with parabolic edge groups.
 \begin{theorem}[Theorem \ref{homeotype}]
 	Let $Y$ be a finite connected graph and let $G(Y),G'(Y)$ be two graph of groups satisfying the following:
 	\begin{enumerate}
 		\item For each vertex $v\in V(Y)$, let $(G_v,\mathbb{P}_v),(G_v',\mathbb{P}_v')$ be relatively hyperbolic groups.
 		\item Let $e\in E(Y)$ be an edge with vertices $v,w$ and let $P_e,P_e'$ are parabolic edge groups in $G(Y),G'(Y)$, respectively. Then either $P_e,P_e'$ have infinite index in corresponding maximal parabolic subgroups in $G_v,G_v'$, respectively or $P_e,P_e'$ have the same finite index in maximal parabolic subgroups in $G_v,G_v'$, respectively. Similarly, either $P_e,P_e'$ have infinite index in maximal parabolic subgroups in $G_w,G_w'$, respectively or $P_e,P_e'$ have the same finite index in corresponding maximal parabolic subgroups in $G_w,G_w'$, respectively.
 		\item Let $B_v$, $B_v'$ be the set of translate of parabolic points corresponding to adjacent edge groups under the action of $G_v,G_v'$ on their Bowditch boundaries respectively. For each vertex $v\in V(Y)$, suppose we have a homeomorphism from $\partial G_v\rightarrow \partial G_v'$ that maps $B_v$ onto $B_v'$.
 	\end{enumerate}
 	Let $\Gamma=\pi_1(G(Y))$, $\Gamma'=\pi_1(G'(Y))$. (By Theorem \ref{rel para}, the groups $(\Gamma,\mathbb{P}),(\Gamma',\mathbb{P'})$ are relatively hyperbolic) Then there exists a homeomorphism from $\partial \Gamma$ to $\partial \Gamma'$ preserving edge parabolic points, i.e. taking parabolic points corresponding to edge groups of $G(Y)$ to parabolic points corresponding to edge groups of $G'(Y)$.
 \end{theorem}
 \begin{rem}
 	The proof of Theorem 0.1 in \cite{dahmni} works only for infinite edge groups. In particular, if the edge groups are finite then the space $M$ constructed in \cite{dahmni} need not be compact. In this paper, in all of the above theorems, we are also taking infinite edge groups. In \cite{alex}, Martin and \'{S}wi\k{a}tkowski constructed Gromov boundary for a graph of hyperbolic groups with finite edge groups, see also \cite{martin1}. If we take a graph of convergence groups with finite edge groups then by \cite{bowrel}, the fundamental group of graph of groups is relatively hyperbolic with respect to infinite vertex groups and hence it a convergence group.
 \end{rem}
 
\textbf{A few words on the proofs}: To prove Theorem \ref{para}, we generalize Dahmani's technique. Here, we explicitly construct a space on which the fundamental group of the graph of groups acts as a convergence group. In \cite{dahmni}, the main role of fully quasi-convex edge groups and acylindrical action was to get uniformly bounded domains (see Definition \ref{domain} ). In our situation, domains for points in edge spaces may be infinite (when edge groups are not maximal parabolic), therefore  the space constructed by Dahmani does not work, See \ref{reason}. Thus, we modify the space. For that, we look at the domain of a point $\xi$ in an edge space and identify all the boundary points (in the visual boundary of Bass-Serre tree) of the domain of $\xi$ with $\xi$ itself. In our situation, domains for points other than edge spaces are singletons. We get a new set by going modulo the equivalence relation generated by this. Then we define a topology on this set and see that this is our candidate space. In Section \ref{paracaseproof}, we prove Theorem \ref{para} as Corollary 4.3.  For proving Theorem \ref{combination}, we use a result of Manning from \cite{manning}. Using this result, proof of this theorem boils down to the parabolic edge group case, and this is done by Theorem \ref{para}. Proof of Theorem \ref{para} also gives proof of Theorem \ref{rel para}. In the proof of Theorem \ref{rel para}, for the construction of candidate space (see Section \ref{3} ), we take Bowditch boundary over which vertex groups act geometrically finitely and the rest of the things remain the same.

 For proving these theorems, it is sufficient to consider only the amalgam and the HNN extension case. For a general graph of groups, we may take a maximal tree in the graph over which we are taking a graph of groups. By proving the theorems for the amalgam, we are done for a graph of groups over a maximal tree. By adding the remaining edges one by one, we are in the HNN extension case. By proving the theorems in the HNN extension case, we are done.

\textbf{Acknowledgement:} I would like to thank my supervisor, Pranab Sardar, for many helpful discussions and comments on the exposition of the paper.
\section{preliminaries}\label{pre}
 We collect here some necessary definitions from \cite{bowconvg},\cite{tukia}, and \cite{bearlimit}
 
 	 A group $\Gamma$ is said to be a {\em convergence group} if there exists a compact metrizable space $M$ and $\Gamma$ acts on $M$ such that: given any sequence $(\gamma_n)_{n \in \mathbb{N}}$ in $\Gamma$, there exists a subsequence $(\gamma_{\sigma(n)})_{n \in \mathbb{N}}$ and two points $\xi$, $\eta$ in $M$ such that  for all compact subset $K$ of $M \setminus \{\eta\}$, $\gamma_{\sigma(n)}K$ converges uniformly to $\xi$.
 	 
 	  Let $\Gamma$ be a convergence group on a compact metrizable space $M$. An infinite order element $\gamma \in \Gamma$ is said to be \emph{loxodromic} if it has exactly two distinct fixed points in $M$. A point $\xi\in M$ is said to be a {\em conical limit point} if there exists a sequence $(\gamma_n)_{n \in \mathbb{N}}$ in $\Gamma$ and two distinct points $\zeta$, $\eta$ such that $\gamma_n\xi$ converges to $\zeta$ and for all $\xi' \in M\setminus \{\xi \}$, $\gamma_n\xi'$ converges to $\eta$.	A subgroup $G$ of $\Gamma$ is said to be \emph{parabolic} if it is infinite, it does not have any loxodromic element and it fixes a point $\xi$. Such a point is unique and called a \emph{parabolic point}. The stabilizer of a unique parabolic point is called {\em maximal parabolic subgroup}. A parabolic point $\xi$ is called a bounded parabolic point if $Stab_{\Gamma}(\xi)$ acts co-compactly on $M\setminus\{\xi\}$. The group $\Gamma$ is said to be {\em geometrically finite} if every point of $M$ is either a conical limit point or a bounded parabolic point.
 	  \begin{rem}
 	  	If $H$ is a subgroup of $\gma$ and  $\gma$ acts as a convergence group on a compact space $M$, then every conical limit point for $H$ action on $\Lambda H \subset M$ (see Definition \ref{defn}(1) below) is a conical limit point for $H$ acting in $M$ and hence for $\gma$ in $M$. Each parabolic point for $H$ in $\Lambda H$ is a parabolic point for $\gma$ in $M$ and its maximal parabolic subgroup in $H$ is exactly the intersection of maximal parabolic subgroup in $\gma$ with $H$.
 	  \end{rem}
Now, we define relatively hyperbolic groups. There are several equivalent definitions of relative hyperbolicity (see \cite{hruska}), but we use the definition given by Bowditch  \cite[Definition 1]{bowrel}.

 We say that a group $\Gamma$ is {\em hyperbolic relative to} a family $\mathcal{G}$ of finitely generated subgroups if $\Gamma$ admits a properly discontinuous isometric action on a proper hyperbolic space $X$ such that induced action of $\Gamma$ on $\partial X$ is convergence, every point of $\partial X$ is either conical or bounded parabolic, and the subgroups in the family $\GG$ are precisely the maximal parabolic subgroups.
 
	In short, we say that the pair $(\Gamma,\GG)$ is a relatively hyperbolic group. In this case, the boundary of $X$ is canonical and called the Bowditch boundary of $(\Gamma,\GG)$. We shall denote it by $\partial \Gamma$. We have the following topological characterization for relatively hyperbolic group given by Yaman \cite{yaman}. 
\begin{theorem}\label{thm1}
\cite{yaman} Let $\Gamma$ be  a geometrically finite group acting on a non-empty perfect metrizable compactum. Assume that the quotient of bounded parabolic points is finite under the action of $\Gamma$ and the corresponding maximal parabolic subgroups are finitely generated. Let $\GG$ be the family of maximal parabolic subgroups. Then $(\Gamma,\GG)$ is a relatively hyperbolic group and $M$ is equivariantly homeomorphic to Bowditch boundary of $\Gamma$. 
\end{theorem}
\begin{rem}
	The assumption that maximal parabolic subgroups are finitely generated does not play any role in the proof of the above theorem, but it is there merely to satisfy the hypothesis in Bowditch's definition of a relatively hyperbolic group. Also by a result of Tukia \cite[Theorem 1B]{pekka}, one can remove the assumption of finiteness of the set of orbits of bounded parabolic points.  
\end{rem}
\begin{defn}\label{defn} We collect the following definitions:
\begin{enumerate}
	\item (Limit set of a subgroup \cite{dahmni}) Let $\gma$ be a convergence group on $M$. The {\em limit set} $\Lambda(H)$ of  an infinite subgroup $H$ is the unique minimal non-empty closed $H$-invariant subset of $M$. The limit set of a finite set is empty.
	
	\item (Relatively quasi-convex subgroup \cite{dahmni}) Let $\Gamma$ be a relatively hyperbolic group with Bowditch boundary $\partial\Gamma$. Let $H$ be a group acting as a geometrically finite convergence group on a compact metrizable space $\partial H$. We assume that $H$ embeds in $\Gamma$ as a subgroup. $H$ is \emph{relatively quasi-convex} if $\Lambda(H) \subset \partial\Gamma$ is equivariantly homeomorphic to $\partial H$.
	
	\item (Dynamically quasi-convex subgroup \cite{bowconvg})  Let $G$ be a convergence group on a compactum $M$. A subgroup $H$ of $G$ is said to be \emph{dynamical quasi-convex} if the following set
	
	$ \{gH\in G/H : g\Lambda(H)\cap K\neq \emptyset, g\Lambda(H)\cap L \neq \emptyset\}$ is finite, whenever $K$ and $L$ are closed disjoint subsets of $M$.
	
	\item (Dynamically malnormal family) Let $\Gamma$ be a convergence group on a compactum $M$. A subgroup $H$ is said to be \emph{dynamically malnormal} if for all $g\in G\setminus H$, $g\Lambda(H)\cap \Lambda(H) =\emptyset$. A collection of subgroups $\{H_i\}_{i\in I}$ is said to form a \emph{dynamically malnormal family} if all $H_i$ are dynamical malnormal and for all $g\in \Gamma$, $g\Lambda(H_i)\cap \Lambda(H_j)=\emptyset$ unless $i=j$ and $g\in H_i$.
	\end{enumerate}
\end{defn}
	Since the limit set of a parabolic subgroup in a convergence group is a singleton, parabolic subgroups are dynamically quasi-convex. Also, by \cite{gerasomov-potya}, relatively quasi-convex subgroups of a relatively hyperbolic group are dynamically quasi-convex. Now, we prove the following:
	
	\begin{lemma}
		Let $M$ be a compact metrizable space. Suppose a group $G$ acts on $M$ as a convergence group. Then infinite cyclic subgroups of $G$ are dynamically quasi-convex.
	\end{lemma}
First of all, we record the following proposition:
\begin{prop}\label{prop1}
	Let $G$ be a group having a convergence action on a compact metrizable space $M$. Let $H$ be a subgroup of $G$ and $\Lambda H$ be its limit set. Then $H$ is dynamically quasi-convex if and only if for any sequence $(g_n)$ in distinct left cosests of $H$ in $G$ there exists a subsequence $(g_{\sigma(n)})$ of $(g_n)$ such that $g_{\sigma(n)}\Lambda H$ uniformly converges to a point. 
\end{prop}
For relatively hyperbolic groups, the above proposition also appears in \cite[Proposition 1.8]{dahmni}. We skip proof of the above proposition as it follows directly from the definition of dynamically quasi-convex subgroup. For the definition of hyperbolic space and Gromov boundary, one is referred to \cite{bridsonhaefliger}. Before proving Lemma 2.5, we prove the following:
\begin{lemma}\label{cyclicdyna}
	Let $X$ be a proper geodesic hyperbolic space and let $G$ be a group acting by isometries on $X$. Suppose $G$ acts as a convergence group on $\partial X$($\partial X$ denotes Gromov boundary of $X$). Then, infinite cyclic subgroups of $G$ are dynamically quasi-convex.
\end{lemma}
\begin{proof}
	Let $g\in G$ be an infinite order element. If $g$ is a parabolic element for $G\curvearrowright \partial X$. Then, clearly $\langle g\rangle$ is dynamically quasi-convex. Suppose $\Lambda(\langle g\rangle)=\{x,y\}$. Suppose $\langle g\rangle$ is not dynamically quasi-convex. Then, by the above proposition, there exists a sequence $(g_n)\subset G\setminus \langle g\rangle$ such that $g_n\Lambda(\langle g\rangle)$ converges to two distinct points $\xi,\eta$ (say). WLOG, assume that $g_nx\rightarrow \xi$ and $g_ny\rightarrow \eta$. Since the action of $G$ on $\partial X$ is convergence, there exists a subsequence $(g_{n_k})$ of $(g_n)$ and two points $a,b\subset \partial X$ such that for all $z\in \partial X\setminus\{b\}$, $g_{n_k}z$ converges to $a$. Note that $G$ also acts on $X\cup \partial X$ as a convergence group with the same attracting and repelling points. If $x\neq b$ and $y\neq b$ then both $g_{n_k}x,g_{n_k}y$ converge to $a$. This implies $\xi=\eta$, a contradiction. Now, suppose $x\neq b$ and $y=b$. Then $g_{n_k}x\rightarrow a$ and $g_{n_k}y \rightarrow \eta$. Thus, $a=\xi$. Now, all the points, except $y$, on bi-infinite geodesic ray joining $x$ and $y$ converge to $a$ under the action of $g_{n_k}$. Now, choose a point $p$ on bi-infinite geodesic ray joining $\xi$ and $\eta$ close enough to $\eta$ and consider a ball around $p$ of radius $R$ for some $R> 0$. Then this ball does not intersect the bi-infinite geodesic rays joining $g_{n_k}x$ to $g_{n_k}y$  for sufficiently large $k$. This is a contradiction as $g_nx$ and $g_ny$ converging to two different points in the boundary of a proper hyperbolic space. We get a similar contradiction when $x=b$ and $y\neq b$. Hence $\langle g \rangle$ is dynamically quasi-convex for $G\curvearrowright \partial X$.
\end{proof}
{\bf Proof of Lemma 2.5:} If $G$ acts on $M$ as an elementary convergence group, then every infinite cyclic subgroup is dynamically quasi-convex. Now, suppose $G$ acts on $M$ as a non-elementary convergence group. Let $Q$ be the set of all distinct triples. Then, by \cite{bowtopo},\cite[Proposition 6.4]{binsun}, there is a $G$-invariant hyperbolic path quasi-metric on $Q$. Thus, we can define the boundary, $\partial Q$, of $Q$ as in \cite{bowtopo}. By \cite[Proposition 4.7]{bowtopo}, $\partial Q$ is $G$-equivariantly homeomorphic to $M$. Now $\partial Q$ is compact as $M$ is compact. Then, by \cite[Proposition 4.8]{bowtopo}, there is a locally finite hyperbolic path quasi-metric space $Q'$ (quasi-isometric to $Q$) and $M$ is $G$-equivariantly homeomorphic to $\partial Q'$. As $Q'$ is locally finite path quasi-metric space, by \cite[Section 3]{bowtopo}, $Q'$ is $G$-equivariantly quasi-isometric to a locally finite graph $G_r(Q')$ for some $r\geq 0$. As $Q'$ is a locally finite hyperbolic path quasi-metric space, $G_r(Q')$ is a proper hyperbolic geodesic metric space. Also, $\partial Q'$ is $G$-equivariantly homeomorphic to $\partial G_r(Q')$. Finally, by the above discussion, $M$ is $G$-equivariantly homeomorphic to $\partial G_r(Q')$. Let $\phi$ be the homeomorphism induced by quasi-isometry from $M$ to $G_r(Q')$. Since $G$ acts on $M$ as a convergence group, $G$ also acts on $\partial G_r(Q')$ as a convergence group \cite{bowconvg}. Also, $G$ acts as a convergence group on $G_r(Q')\cup \partial G_r(Q')$. Suppose $g\in G$ such that order of $g$ is infinite. If $g$ is a parabolic element. Then, clearly $\langle g\rangle$ is dynamically quasi-convex for $G\curvearrowright M$. Now, suppose $g$ loxodromic for the action of $G$ on $M$. Then, $g$ is loxodromic for the action of $G$ on $\partial G_r(Q')$. By the lemma 2.7, $\langle g\rangle$ is dynamically quasi-convex subgroup of $G$ for $G\curvearrowright \partial G_r(Q')$. Now, let if possible $\langle g\rangle$ is not dynamically quasi-convex for $G\curvearrowright M$. Then there exist disjoint closed subsets $K,L$ of $M$ such that the set $\{g\in G\setminus \langle g\rangle| g\Lambda(\langle g\rangle)\cap K \neq \emptyset, g\Lambda(\langle g\rangle)\cap L \neq \emptyset\}$ is infinite. Since $\phi$ is a homeomorphism, $\phi(K),\phi(L)$ are disjoint closed subsets of $\partial G_r(Q')$. Then, by $G$-equivariance of $\phi$, $\langle g \rangle$ is not dynamically quasi-convex for $G\curvearrowright \partial G_r(Q')$. This gives us a contradiction.   \qed
%%%%%%%%%%%%%%%%%%%%%%%%%%%%%%%%%%%%%%%%%%%%%%%%%%%%%%%%%%%%%%%%%%%%%%%%%%%%%%%%%%%%%%%%%%%%%%%%%%%%%%%%%%%%%%%%%%%%%%%%%%%%%%
\section{Construction of boundary of graphs of groups}\label{3}
As we mentioned in the introduction, it is sufficient to consider the amalgam and the HNN extension case for proving our theorems. This will be our standing assumption for the next two sections. Let $\gma$ be an amalgamated free product or an HNN extension of convergence groups along with parabolic edge group. Let $T$ be the Bass-Serre tree of this splitting and let $\tau$ be a subtree of $T$, an edge in case of amalgam, and a vertex in case of HNN extension. Here, we construct our candidate space $M$ on which $\Gamma$ acts as a convergence group.

We fix some notation: If $v$ is a vertex of $T$, we write $\gma_v$ for its stabilizer in $\gma$. Similarly, for an edge $e$, we write $\gma_e$ for its stabilizer. For each vertex $v$ and edge $e$ incident on $v$, $\gma_v$ is a convergence group, and $\gma_e$ is a parabolic subgroup in $\Gamma_v$. We denote $X_v$ and $X_e$ as compact metrizable spaces on which the vertex group $\Gamma_v$ and the edge group $\Gamma_e$ act as convergence groups. In our situation, $X_e$ is singleton.

\subsection{Definition of $M$ as a set}\hfill\\

\textbf{Contribution of the vertices of $T$}

Let $\VV(\tau)$ be the set of vertices of $\tau$. For a vertex $v\in\VV(\tau)$, the group $\gma_v$ is by assumption a convergence group. So, we have compactum (compact metrizable space) $X_v$ for $\Gamma_v$. Set $\Omega$ to be $\gma \times(\displaystyle\bigsqcup_{v\in \VV(\tau)}X_v)$ divided by the natural relation\\
 \[(\gamma_1,x_1) = (\gamma_2,x_2) \text{ if }  \exists v\in \VV(\tau), x_i\in X_v,\gamma_2^{-1}\gamma_1 \in \gma_v, (\gamma_2^{-1}\gamma_1)x_1 =x_2 \text{ for } i=1,2\]
 
  In this way, $\Omega$ is the disjoint union of compactums corresponding to the  stabilizers of the vertices of $T$. Also, for each $v\in \VV(\tau)$, the space $X_v$ naturally embeds in $\Omega$ as the image of $\{1\}\times X_v$. We identify it with its image. The group $\gma$ naturally acts on the left on $\Omega$. For $\gamma\in \Gamma$, $\gamma X_v$ is the compactum for the vertex stabilizer $\gma_{\gamma v}$.
 
 \textbf{Contribution of the edges of $T$} 
 
 Each edge allows us to glue together compactums corresponding to the vertex stabilizers along with the limit set of the stabilizer of the edge. Each edge group embeds as a parabolic subgroup in adjacent vertex groups, so its limit set is a singleton. Let $e$ = $(v_1,v_2)$ be the edge in $\tau$, there exist equivariant maps $\Lambda_{e,v_i}:X_e \hookrightarrow X_{v_i}$ for $i=1,2$. Similar maps are defined by translation for edges in $T\setminus \tau$.
 
 The equivalence relation $\sim$ on $\Omega$ is the transitive closure of the following: Let $v$ and $v'$ be vertices of $T$. The points $x \in X_v$ and $x' \in X_{v'}$ are equivalent in $\Omega$ if there is an edge $e$ between $v$ and $v'$ and a point $x_e \in X_e$ satisfying $x = \Lambda_{e,v}(x_e)$ and $x' = \Lambda_{e,v'}(x_e)$ simultaneously. Let $\Omega/_{\sim}$ be the quotient under this relation and let $\pi':\Omega \map \Omega/_{\sim}$  be the corresponding projection. An equivalence class $[x]$ of an element $x\in \Omega$ is denoted by $x$ itself. \\ Let $\partial T$ be the (visual) boundary of the tree. We define $M'$ as a set: $M' = \partial T \sqcup (\Omega/_{\sim})$. 
 \begin{defn}\label{domain}
 	$\bf{(Domains)}$ For all $x \in \Omega/_{\sim}$, we define the {\em domain} of $x$ to be $D(x) = \{v\in \VV(T) | x\in \pi'(X_v)\}$. We also say that domain of a point $\xi\in \partial T$ is $\{\xi\}$ itself.
 \end{defn}
\subsection{Final construction of $M$}
 It turns out that $M'$ with the topology defined in \cite{dahmni} is not a Hausdorff space (see \ref{reason} ). Therefore we need to modify $M'$ further. For getting the desired space, we  further define an equivalence relation on the set $M'$. Firstly observe that the domain of each element in $\omg$ is either singleton or an infinite subtree of $T$. So, we also identify the boundary points of infinite domain $D(x)$ in $\partial T$ to $x$ itself. By considering the equivalence relation generated by these relations, we denote the quotient of $M'$ by $M$. Again, $M$ can be written as a disjoint union of two sets of equivalence classes :
 \begin{center}
 	$M = \Omega' \sqcup (\partial T)'$
 \end{center} 
 where $\Omega'$ (as a set it is same as $\omg$) is the set of equivalence classes of elements in $\omg$, and, $(\partial T)'$ is the equivalence classes of the remaining elements in $\partial T$ as some elements of $\partial T$ are identified with parabolic points of edge groups. The equivalence class of each remaining element in $\partial T$ is a singleton. We write $p$ for an element of $M$, if $p\in \Omega'$ then $p=x,y,z,...$ and if $p\in (\partial T)'$ then $p=\xi,\eta,\zeta,...$.
 \begin{rem}
 	In $M$, we define domain of each element as previously defined. Additionally for those points $\eta$ of $\partial T$ which are identified with parabolic  points $x$ of edge groups, we define domain of $\eta$ same as that of $x$. 
 \end{rem}
 It is clear that for each $v$ in $T$, the restriction of projection map $\pi'$ from  $X_v$ to $\omg$ is injective. Let $\pi''$ be the projection map from $M'$ to $M$. Let $\pi$ be the composition of the restriction of $\pi''$ to $\omg$ and $\pi'$. Again it is clear that restriction of $\pi$ to $X_v$ is injective for all $v$ in $T$.
\subsection{Definition of neighborhoods in M} \hfill\\

 We define a family $(W_n(p))_{n\in \mathbb{N},p\in M}$ of subsets of $M$, which generates a topology on $M$. For a vertex $v$ and an open subset $U$ of $X_v$, we define the subtree $T_{v,U}$ of $T$ as $\{w\in \VV(T) : X_e\cap U\neq \emptyset\}$, where e is the first edge of $[v,w]$. For each vertex $v$ in $T$, let us choose $\UU(v)$, a countable basis of open neighborhoods of $X_v$. Without loss of generality, we can assume that for all $v$, the collection of open subsets $\UU(v)$ contains $X_v$. Let $x$ be in $\Omega'$ and $D(x) = \{v_1,...,v_n,...\} = (v_i)_{i\in I}$. Here $I$ is a subset of $\mathbb{N}$. For each $i\in I$, let $U_i \subset X_{v_i}$ be an element of $\UU(v_i)$ containing $x$ such that for all but finitely many indices $i\in I, U_i=X_{v_i}$. The set $W_{(U_i)_i\in I}(x)$ is the disjoint union of three subsets 
 \begin{center}
 	$W_{(U_i)_{i\in I}}(x) = A \cup B\cup C$,
 \end{center} 
 
 where $A$ is nothing but the collection of all boundary points of subtrees $T_{v_i,U_i}$ which are not identified with some parabolic point corresponding to edge group, $B$ is collection of all points $y$ outside $\bigcup_{i\in I}X_{v_i}$ in $\Omega'$ whose domains lie inside $\bigcap_{i\in I}T_{v_i,U_i}$, $C$ is simply the union of all neighborhood $U_i$ around $x$ in each vertex of $D(x)$. In notations $A$,$B$,$C$ are defined as follows:
\begin{align*}
A &= (\bigcap_{i\in I} \partial T_{v_i,U_i}) \cap (\partial T)'\\
B &= \{y \in \Omega'\setminus(\bigcup_{i\in I}X_{v_i}) | D(y) \subset \bigcap_{i\in I}T_{v_i,U_i}\}\\
C &=\{y \in \bigcup_{j\in I}X_{v_j} | y \in \bigcup_{m\in I | \zeta \in X_{v_m}}U_m\}
\end{align*}

 As $A \subset \pt$, the remaining elements in $\displaystyle\bigcap_{i\in I} \partial T_{v_i,U_i}$ are in $B$. In this way, $A$, $B$, $C$ are disjoint subsets of $M$. 
\begin{rem}
	The set $W_{(U_i)_{i\in I}}(x)$ is completely defined by the data of the domain of $x$, the data of a finite subset $J$ of $I$, and the data of an element of $\UU(v_j)$ for each index $j\in J$. Therefore there are only countably many different sets  $W_{(U_i)_{i\in I}}(x)$, for $x\in \Omega'$, and $U_i\in \UU(v_i),v_i\in D(x)$. For each $x$, we choose an arbitrary order and denote them by $W_m(x)$.
\end{rem}
Now, we  define neighborhoods around the points of $\pt$. Let $\xi\in \pt$ and choose a base point $v_0$ in $T$. Firstly we define the subtree $T_m(\xi)$: it consists of the vertices $w$ such that $[v_0,w] \cap [v_0,\xi)$ has length bigger than $m$. We set $W_m(\xi) = \{\zeta \in M | D(\zeta) \subset T_m(\xi)\}$. This definition does not depend on the choice of base point $v_0$, up to shifting the indices.
\begin{lemma}\label{avoid}
	\emph{(Avoiding an edge)} Let $p$ be a point in $M$ and $e$ an edge in $T$ with at least one vertex not in $D(p)$. Then there exists an integer $m$ such that $W_n(p) \cap X_e = \emptyset$.
\end{lemma}
\begin{proof}
	If $p\in \pt$, then both of the vertices of $e$ are not in $D(p)$. The lemma follows by taking one of the vertices as base and $m \geq1$. If $p\in \Omega'$ then a unique segment exists from $D(p)$ to the edge $e$. Let $v$ be the vertex from where this segment starts and $e_0$ be the first edge. Then $X_{e_0}$ does not contain $p$. So to find a $W_m(p)$ such that it does not intersect $X_e$, it is sufficient to find a neighborhood around $p$ in $X_v$ which does not intersect $X_{e_0}$. But it is evident as $X_{e_0}$ is just a point.  
\end{proof}
\subsection{Topology of M}  \hfill\\

Consider the smallest topology $\TT$ on $M$ such that the family of sets $\{W_n(p): n\in \mathbb{N},p\in M\}$ are open subsets of $M$.
\begin{lemma}
	$(M,\TT)$ is a Hausdorff space
\end{lemma}
\begin{proof}
	Let $p$ and $q$ be two distinct points in $M$. If $D(p) \cap D(q) = \emptyset$. Then there is a unique geodesic segment from a vertex of $D(p)$ to a vertex of $D(q)$ in $T$ having an edge $e$ on this segment such that both vertices of $e$  are neither contained in $D(p)$ nor in $D(q)$. Then using the above lemma, we can find disjoint neighborhoods around $p$ and $q$. Now, suppose $D(p) \cap D(q) \neq \emptyset$. Then there is only one vertex in this intersection, let $D(p) \cap D(q) = \{v\}$. Since $X_v$ is Hausdorff,  we can find disjoint open subsets of $X_v$ around $p$ and $q$, respectively. Using these neighborhoods in $X_v$, we have found disjoint neighborhoods around $p$ and $q$ in $M$.
\end{proof}
\begin{rem}\label{reason}
	The reason why we need to define a further equivalence relation on the set $M'$ is the following:
	
	In $M'$, if we consider an edge group parabolic point $p$, then $D(p)$ is infinite subtree of $T$, and there are uncountably many boundary points of $D(p)$. If we take one boundary point $\eta$ of $D(p)$, then on the geodesic ray $[v_0,\eta)$ there is no edge with at least one vertex not belonging to $D(p)$, where $v_0$ is some vertex in $D(p)$. Thus we can not find a disjoint neighborhood around $\eta$ and $p$. Same kind of situation will arise if we prove that $M'$ is a regular space.
\end{rem}
\begin{lemma}
	\emph{(Filtration)} For every $p\in M$, every integer n, and every $q\in W_n(\xi)$, there exists $m$ such that $W_m(q) \subset W_n(p)$.
\end{lemma}
\begin{proof}
	Suppose $p\in \pt$ and $W_n(p)$ is a neighborhood around $p$. Let $q\in W_n(p)$. If $q\in \pt$ then choose $m=n$ and $W_m(q) = W_n(p)$. Let $q$ be  some point in $\Omega'$. Suppose the subtree $T_n(p)$ starts at the vertex $v$ and $e$ be the last edge on the segment from a base vertex to $v$ then except $X_e$, all the points in $M$ corresponding to subtree $T_n(p)$ is in $W_n(p)$. By this observation, it is clear that there exists a neighborhood $W_m(q)$ such that $W_m(q)\subset W_n(p)$ for sufficiently large $m$. Now, suppose that $p\in \Omega'$ and $q\in W_n(p)$. If $D(p)\cap D(q) = \emptyset$, then there exists an edge $e$ on a unique geodesic segment from $D(p)$ to $D(q)$. Then by Lemma \ref{avoid}, one can find a neighborhood $W_m(q)$ sitting inside $W_n(p)$. Finally, if $D(p)\cap D(q) \neq \emptyset$, then it is a singleton and let this intersection be $\{v\}$. As $q\in W_n(p)$, $q$ is in some $U_i$, where $U_i$ is in  neighborhood basis of $X_{v_i}$ and $D(p) = (v_i)_{i\in I}$. Then find a neighborhood $V_i$ around $q$ sitting inside $U_i$ and then using this $V_i$ we see that we have found $W_m(q)$ such that $W_m(q)\subset W_n(p)$.
\end{proof}
\begin{lemma}
	The family of sets $\{W_n(p) : n\in \mathbb{N},p\in M\}$ forms a basis for the topology $\mathcal{T}$.
\end{lemma}
\begin{proof}
	Using the previous lemma, it remains to prove that if $W_{n_1}(p_1)$ and $W_{n_2}(p_2)$ are two neighborhoods and $q\in  W_{n_1}(p_1) \cap W_{n_2}(p_2)$ then there exists a neighborhood around $q$, namely $W_m(q)$ such that $W_m(q) \subset  W_{n_1}(p_1) \cap W_{n_2}(p_2)$. Again by the previous lemma there exists $m_1$ and $m_2$ such that $W_{m_1}(q) \subset W_{n_1}(p_1)$ and $W_{m_2}(q)\subset W_{n_1}(p_2)$. Now $W_k(q)\subset W_{m_1}(q) \cap W_{m_2}(q)$ for some $k$, hence the lemma.
\end{proof}
\begin{lemma}
For each $v\in \VV(T)$, the restriction of $\pi$ to $X_v$ to $M$ is a continuous map. 
\end{lemma}
\begin{proof}
	Let $x$ be an element of $X_v$, and $\pi(x)$ be its image in $M$. We denote this image by $x$. Consider the neighborhood $W_n(x)$ around $x$. Then by definition of neighborhoods, it is clear that the inverse image of this neighborhood under $\pi$ is an open subset of $X_v$. Hence restriction of $\pi$ to $X_v$ is continuous.
\end{proof}
\begin{lemma}
	\emph{(Regularity)} The topology $\mathcal{T}$ is regular, that is, for all $p\in M$ and for all $W_m(p)$ there exists $n$ such that $\overline{W_n(p)} \subset W_m(p)$.
\end{lemma}

\begin{proof}
	Case(1) Let $p\in \pt$ and $W_m(p)$ be a neighborhood around $p$. Let $v$ be a vertex from where the subtree $T_m(\xi)$ starts. Let $e$ be the last edge of the geodesic segment from $v_0$ to $v$. Observe that closure of $W_m(p)$ contains only one extra point, namely $X_e$. By taking $n$ to be sufficiently large, we see $\overline{W_n(p)}\subset W_m(p)$.
	
	Case(2) Let $p\in \Omega' $ and let $W_m(p)$  be a neighborhood around $p$. Let $D(p) = (v_i)_{i\in I}$. Again observe that in the closure of $W_m(p)$ only extra points are the points in the closure of each $U_i$ in $X_{v_i}$.(If some point is not in the closure of $U_i$ then one can easily find a neighborhood around that point disjoint from $W_m(p)$). Also since each vertex boundary is regular so choose a neighborhood $V_i$ of $\xi$ in $X_{v_i}$ such that $\overline{V_i} \subset U_i$. Then $\overline{W_{\{V_i\}_{i\in I}}(p)}\subset W_m(p)$.
\end{proof}

\textbf{Note:} Now by previous lemmas, we see that the topology on $M$ is second countable, Hausdorff and regular. Then by Urysohn's metrization theorem, we see that $M$ is metrizable space. Also, $M$ is a perfect space as every point of $M$ is a limit point. Finally, we have the following convergence criterion: \label{convergence criterion}

A sequence $(p_n)_{n\in \mathbb{N}}$ in $M$ converges to a point $p$ if and only if $\forall n$, $\exists m_0\in \mathbb{N}$ such that $\forall m > m_0$, $p_m \in W_n(p)$.
\begin{theorem}
	The metrizable space $M$ is compact.
\end{theorem}
\begin{proof}
	It is sufficient to prove that $M$ is sequentially compact. Let $(p_n)_{n\in \mathbb{N}}$ be a sequence in $M$. Let us choose  a vertex $v$ in $T$ and for each $n$, choose a vertex $v_n$ (if $p_n\in \pt$ then choose $v_n = p_n$ ) in $D(p_n)$. We see that up to extraction of a subsequence, either the Gromov inner products $(v_n,v_m)_v$ remain bounded or they go to infinity. In the latter case $(v_n)$ converges to a point $q$ in $\partial T$. If $q$ is the point which is identified with some edge boundary point then there is a ray in the domain of that edge boundary point converging to $q$, then by definition of the neighborhood around $q$, we see that $p_n$ converges to $q$. If $q\in\pt$ then again there is a ray converges to $q$, and by convergence criterion, $p_n$ converges to $q$. Now, in the first case up to extracting a subsequence, we assume that Gromov inner products is equal to some constant $N$. Let $g_n$ be the geodesic from $v$ to $v_n$ then there exist  geodesic $g = [v,v']$ of length $N$ such that $g$ lies in each $g_n$. For different $n$ and $m$, $g_n$ and $g_m$ do not have a common prefix whose length is longer than $g$. We have the following two cases:
	
	Case(1): There exist a subsequence $(g_{n_k})$ such that $g_{n_k} = g$. Then in this case as $X_v'$ is compact; we get a subsequence of $(p_n)_{n\in \mathbb{N}}$ which converges to a point of $X_v'$.
	
	Case(2): There exists a subsequence $(g_{n_k})$ of $(g_n)$ such that each $g_{n_k}$ strictly longer than $g$. Let $e_{n_k}$ be the edges just after $v'$; they all are distinct. As for each edge $e$, $X_e$ is a singleton, so $(X_{e_{n_k}})$ forms a sequence of points in $X_{v'}$. Since $X_{v'}$ is compact, there exists a subsequence which converges to a point in $X_{v'}$. Then, by convergence criterion, we see that there exists a subsequence of $(p_n)$ converging to this point of $X_{v'}$. (It may be possible that all the $X_{n_k}$ are equal to $p$ for some $p$ in $X_{v'}$, sequence converges to $p$ in this situation also.)
	
\end{proof}

\section{Dynamics of $\gma$ on $M$} \label{paracaseproof}

In this section, we prove that $\gma$ acts on $M$ as a convergence group. For that we need the following two lemmas.
\begin{lemma}
	\emph{(Large Translation)}\label{large} Let $(\gamma_n)_{n\in \mathbb{N}}$ be a sequence in $\gma$. Assume that, for some (hence any) vertex $v_0\in T$, $dist(v_0,\gamma_nv_0) \rightarrow \infty$. Then, there is a subsequence $(\gamma_{\sigma(n)})_{n\in \mathbb{N}}$, two points $p\in M$ and $\zeta\in \pt$ such that for all compact $K\subset (M\setminus \{\zeta\})$, $\gamma_{\sigma(n)} K$ converges uniformly to $p$. 
\end{lemma}
\begin{proof}
	Let $p_0$ be in $X_{v_0}$. Since $M$ is sequentially compact, there exists a subsequence $\sub$ such that $(\gamma_{\sigma(n)}p_0)_{n\in\mathbb{N}}$ converges to a point $p\in M$. Also, we still have $dist(v_0, \gamma_{\sigma(n)}v_0) \rightarrow \infty$. Let $v_1$ be another vertex in $T$. The lengths of the segments $[\gamma_nv_0, \gamma_nv_1]$ equal to the length of $[v_0, v_1]$. As the $dist(v_0, \gamma_{\sigma(n)}v_0)$ goes to $\infty$, for all $m$ there is $n_m$ such that for all $n > n_m$, the segments $[v_0, \gamma_{\sigma(n)}v_0]$ and $[v_0, \gamma_{\sigma(n)}v_1]$ have prefix of length more than $m$. Then, by convergence criterion \ref{convergence criterion}, for all $v\in T$, $\gamma_{\sigma(n)}X_v$ converges uniformly to $p$. Let $\zeta_1,\zeta_2 \in \pt$. As triangles in $\overline{T}$ are degenerate, so the triangle with vertices $v_0,\zeta_1,\zeta_2$ have center some vertex $v$ in $T$. Therefore for all $m\geq 0$, the segments $[v_0, \gamma_{\sigma(n)}v_0]$ and $[v_0, \gamma_{\sigma(n)}v]$ coincide on a subsegment of length more than $m$ for sufficiently large $n$. Then at least for one $\zeta_i$, the sequence of rays $[v_0,\gamma_{\sigma(n)}\zeta_i]$ have a common prefix of length at least $m$. Then by convergence criterion $\gamma_{\sigma(n)}\zeta_i$ converges to $p$. Therefore, there exists at most one (since for any two points in $\pt$ at least one of them converges to $p$ under the action of $\gamma_n$) $\zeta$ in $\pt$ such that for all $\zeta'\in (\pt\setminus\{\zeta\})$, we have $\gamma_{\sigma(n)}\zeta'\rightarrow p$.
	
	Let $K$ be a compact subspace of $M\setminus \{\zeta\}$. Then there exists a vertex $v$, $x\in \Omega'$, and a neighborhood $W_n(x)$ of $x$ containing $K$ such that $\zeta \notin W_n(x)$. Consider the segment $[v_0,v]$. As by the discussion at the beginning of proof, the sequence $(\gamma_{\sigma(n)}X_v)_{n\in \mathbb{N}}$ uniformly converges to $p$, the sequence $(\gamma_{\sigma(n)}W_m(x))_{n\in \mathbb{N}}$ uniformly converges to $p$. Hence the convergence is uniform on $K$.
\end{proof}
\begin{lemma}
	\emph{(Small Translation)}\label{small} Let $(\gamma_n)_{n\in \mathbb{N}}$ be a sequence of distinct elements in $\gma$, and assume that for some (hence any) vertex $v_0$, the sequence $(\gamma_n v_0)_n$ is bounded in $T$. Then there exists a subsequence $\sub$, a vertex $v$, a point $p\in X_v$, and another point $p'\in \Omega'$, such that for all compact subspaces $K$ of $M\setminus \{p'\}$, one has $\sub K \rightarrow p$ uniformly. 
\end{lemma}
\begin{proof}
	We prove the lemma in two cases.
		
	Case(1): Assume that for some vertex $v$ and some element $\gamma\in\gma$ there exists a subsequence $(h_n)_{n\in \mathbb{N}}$ in $\gma_v$ such that $\gamma_n = h_n\gamma$ for all $n$. Since $\gma_v$ acts as a convergence group on $X_v$. Then one can further extract a subsequence of $(\gamma_n)$ (we shall denote it again by $\gamma_n$) and a point $p'$ in $X_{\gamma^{-1}v}$ such that for all compact subsets $K$ of $X_{\gamma^{-1}v}\setminus\{p'\}$, $\gamma_nK\rightarrow p$ uniformly, for some $p\in X_v$. Suppose $p'$ is a conical limit point. then it will not be in any edge boundary contained in $X_{\gamma^{-1}v}$. Let $e_n$ be the possible edges starting from the vertex $\gamma^{-1}v$. For any  $q \in  M\setminus\{X_{\gamma^{-1}v}\}$, we see that unique edge path from $\gamma^{-1}v$ to $w$ contains $e_n$, for some n and $q\in X_w$. Since $\gamma_nr\rightarrow p$ for all $r$ in $X_{\gamma^{-1}v}\setminus\{p'\}$ then by convergence criterion, we see that $\gamma_n q$ converges to $p$, and the same is true for the points in $\pt$. Hence for all compact $K \subset M\setminus\{p'\}$ we have $\gamma_nK \rightarrow p$ uniformly. Suppose $p'$ is a parabolic point, $\gamma p'$ is also a parabolic point in $X_v$. Then $\gamma_n p'$ also converges to $p$. Otherwise, $\gamma p'$ is a conical limit point. Thus again, by the same argument as above we see that for all $q\in M$, $\gamma_nq \rightarrow p$.
	
	Case(2): Suppose such a sequence $(h_n)_{n\in \mathbb{N}}$ and a vertex $v$ does not exist. After possible extraction we can assume that the distance $dist(v_0,\gamma_n(v_0))$ is constant. Let us choose a vertex $v$ such that there exists a subsequence $\sub$ such that the segments $[v_0,\gamma_{\sigma(n)}{v_0}]$ share a common segment $[v_0,v]$ and the edges $e_{\sigma(n)}$ located just after $v$ are all distinct. Then one can extract a subsequence $(e_{\sigma'(n)})_{n\in \mathbb{N}}$ such that spaces corresponding to these edges converges to a point $p$ in $X_v$. Then by convergence criterion, $\gamma_{\sigma'(n)}(X_{v_0})$ converge uniformly to $p$. Let $\xi\in \pt$. Then $v$ is not on the ray $[\gamma_{\sigma'(n)}v_0, \gamma_{\sigma'(n)}\xi]$ for sufficiently large $n$ for if $v$ is there then, for infinitely many $n$, $\gamma^{-1}_{\sigma'(n)}v = w$ for some fixed vertex $w$ on $[v_0,\xi)$, we see that we are in the first case, which is a contradiction. Thus for all $\xi\in \pt$, we see that the unique rays from $v$ to $\gamma_{\sigma'(n)}\xi$ contains the edge $e'_{\sigma(n)}$ for some $n$. Hence by convergence criterion $\gamma_{\sigma'(n)}\xi \rightarrow p$ for all $\xi \in \pt$. Now, let $x$ be in the space except $X_{v_0}$. Suppose $x\in X_{v'}$. Again by the same reasoning, we see that $v\notin [\gamma_{\sigma'(n)}v_0, \gamma_{\sigma'(n)}v']$ for sufficiently large $n$. Then the unique geodesic segment from $v$ to $\gamma_{\sigma'(n)}v'$ contains $e'(n)$ for sufficiently large $n$. Hence by definition of neighborhoods $\gamma_{\sigma'(n)}x \rightarrow p$. Thus, for all compact subsets $K \subset M$, we see that $\gamma_{\sigma'(n)}K \rightarrow p$ uniformly.
	
\end{proof}
Using the previous two lemmas, we get the following:
\begin{cor}
	The group $\gma$ acts on $M$ as a convergence group.
\end{cor}
\begin{proof}
	Fix a vertex $v_0$ in $T$. Let $(\gamma_n)_{n\in\mathbb{N}}$ be a sequence in $\Gamma$. Then up to extraction of a subsequence, there are two cases: either the distance from $v_0$ to $\gamma_nv_0$ goes to infinity or the distance from $v_0$ to $\gamma_nv_0$ is bounded. In either case, previous lemmas imply the corollary, which proves Theorem \ref{para}.
\end{proof}
\section{Combination of convergence groups} In this section, we prove Theorem \ref{combination}. First of all, we recall a construction from \cite{manning} that is used in the proof of Theorem \ref{combination}. Let $\Gamma$ be a hyperbolic group and let $(\Gamma,\mathcal{G})$ be a relatively hyperbolic group. In \cite{manning}, Manning constructed a space which is the quotient of Gromov's boundary of $\Gamma$ and showed that $\Gamma$ acts geometrically finitely on the quotient. The quotient is obtained by collapsing all the translates of the limit set of subgroups in $\mathcal{G}$. Using Yaman's characterization (Theorem \ref{thm1}) of relative hyperbolicity, this quotient is Bowditch boundary of $(\Gamma,\mathcal{G})$. There, to prove that the action of $\Gamma$ on the quotient is convergence, we do not require that $\Gamma$ is hyperbolic and $\mathcal{G}$ is a malnormal family of quasi-convex subgroups. 

Suppose $\Gamma$ acts on a compact metrizable space $X$ as a convergence group. Let $\mathcal{G}$ is a dynamically malnormal family of dynamically quasi-convex subgroups. Then form a quotient space $X/\sim$ as in \cite{manning} by collapsing translates of limit sets of subgroups in $\mathcal{G}$. Also, assume that $|\Lambda(H)|\geq 2$, where $H\in \mathcal{G}$. To prove Proposition 2.2 in \cite{manning}, we require that the collection of limit sets of the cosets of the elements in $\mathcal{G}$ form a null sequence which follows from Proposition \ref{prop1}. We immediately have the following lemma:
\begin{lemma}\label{manning}
	\begin{enumerate}
	 	\item $X/\sim$ is a compact metrizable space.
	 	\item The group $\Gamma$ acts on $X/\sim$ as a convergence group.
	 \end{enumerate}
\end{lemma}
Also, note that each subgroup in $\mathcal{G}$ become a parabolic subgroup for the action of $\Gamma$ on $X/\sim$. Now, we give proof of Theorem \ref{combination}.

\subsection*{Proof of \thmref{combination}}\label{combination proof} Let $\Gamma$ be as in the statement of the theorem. Let $\Gamma_v$ be a vertex group that acts as a convergence group on $X_v$. Take the collection of edges incident to vertex $v$ and take a collection of those edge groups which are not parabolic in the vertex group $\Gamma_v$. Consider the stabilizers of the limit sets of these edge groups in the vertex group $\Gamma_v$. By assumption, they form a dynamical malnormal family of dynamically quasi-convex subgroups. Therefore, by Lemma \ref{manning}, we obtain a quotient of $X_v$, namely $X_{v}/\sim$ on which $\Gamma_v$  acts as a convergence group and all edge groups incident to $v$  become parabolic subgroups of $\Gamma_v$. Hence by following the same process at each vertex group, we are in the situation where we have a graph of convergence groups with edge groups parabolic in adjacent vertex groups. This completes the proof using Theorem \ref{para}.  \qed
\subsection*{Proof of \propref{cyclic}} It is sufficient to consider the amalgam and HNN extension case.

Case(1): Let $\Gamma=\Gamma_1 \ast_{\langle \gamma_1\rangle \simeq \langle \gamma_2\rangle}\Gamma_2$. Suppose $\Gamma_1$ is acting on $X_1$ as a convergence group and $\gamma_1$ is a loxodromic element in $\Gamma_1$. Then, by \cite[Lemma 2.6]{yang},  stabilizer of the limit set of $\langle \gamma_1\rangle$ in $\Gamma_1$ is a dynamically quasi-convex subgroup of $\Gamma_1$. Also, by assumption, the stabilizer of the limit set of $\langle \gamma_1\rangle$ in $\Gamma_1$ is dynamically malnormal. Similarly, the stabilizer of the limit set of $\langle \gamma_2\rangle$ in $\Gamma_2$ is dynamically quasi-convex and dynamically malnormal. Hence, by Theorem \ref{combination}, $\Gamma$ is a convergence group.

Case(2): Let $\Gamma=\Gamma_1\ast_{\langle\gamma_1\rangle \simeq \langle\gamma_2\rangle}$. Suppose $\Gamma_1$ acts as a convergence group $X_1$ and $\gamma_1$ is loxodromic for this action. Then, by Lemma \ref{manning}, $\Gamma_1$ acts as a convergence group on the quotient $X_1/\sim$ of $X_1$, $\gamma_1$ is a parabolic element for the action of $\Gamma_1$ on $X_1/\sim$ . Now, if $\gamma_2$ is parabolic for the action of $\Gamma_1$ on $X_1/\sim$ then we have HNN extension of convergence group with parabolic edge group, and hence by Theorem \ref{para}, $\Gamma$ is a convergence group. If $\gamma_2$ is loxodromic for the action of $\Gamma_1$ on $X_1/\sim$ then again, by Lemma \ref{manning}, we have the quotient of $X_1/\sim$ such that $\Gamma_1$ acts as a convergence group. $\gamma_2$ is a parabolic element for the action of $\Gamma_1$ on the quotient of $X_1/\sim$. Hence, we have the HNN extension of convergence group with parabolic edge group, and, by Theorem \ref{para}, $\Gamma$ is a convergence group.         \qed

{\bf Note:}\label{reason 1} The amalgam case in the above corollary is exactly a corollary of Theorem \ref{combination}. However, the HNN extension is not exactly a corollary of Theorem \ref{combination} as stabilizers of limit sets of $\langle \gamma_1\rangle$, $\langle \gamma_2\rangle$ in $\Gamma_1$ respectively need not form a dynamical malnormal family.

When the vertex groups in Corollary \ref{cyclic} are torsion-free, then, by the following lemma, we do not need to assume dynamical malnormality of edge group in adjacent vertex groups.
\begin{lemma}\label{dyna}
	Let $\Gamma$ be a torsion-free group that acts on $X$ as a convergence group. Let $\gamma\in \Gamma$ be a loxodromic element and let $H=Stab_{\Gamma}(\Lambda(\langle\gamma\rangle))$ Then $H$ is a dynamically quasi-convex and dynamically malnormal subgroup of $\Gamma$.
\end{lemma}
\begin{proof}
	The dynamical quasi-convexity of $H$ follows from Lemma 2.6 of \cite{yang}. Let $\gamma_1\in \Gamma\setminus H$ and assume that $\gamma_1\Lambda(H)\cap \Lambda(H)\neq \emptyset$. Let $\Lambda(H)=\{x_1,x_2\}$ and let $\gamma_1$ fixes, either $x_1$ or $x_2$. Since $\Gamma$ is torsion-free, $\gamma_1$ has infinite order. Thus $\gamma_1$ is either a parabolic or a loxodromic element. By \cite[Proposition 3.2]{bowconvg}, a parabolic point can not be a fixed point of a loxodromic element so $\gamma_1$ can not be parabolic. By \cite[Theorem 2G]{tukiametric}, if $\gamma_1$ is loxodromic, then $\gamma_1$ must fix the other point. But this implies that $\gamma_1$ is in $H$, which is a contradiction. Now, suppose that $\gamma_1$ does not fix any of $x_i$ and $\gamma_1x_1 =x_2$. Consider the element $\gamma_1' = \gamma_1^{-1}\gamma \gamma_1$ which fixes $x_1$. Again $\gamma_1'$ can not be parabolic so it has to be loxodromic, but this implies that $\gamma_1' x_2=x_2$. Thus $\gamma_1x_2 =x_1$ and this implies that $\gamma_1$ is in $H$, which is again a contradiction. 
\end{proof}
Hence for torsion-free groups, we have the following:
\begin{prop}\label{torsionfreecyclic}
	Let $\Gamma$ be the fundamental group of a finite graph of torsion-free countable convergence groups with infinite cyclic edge groups. Then $\Gamma$ is a convergence group.    \qed
\end{prop} 
\begin{rem}
	Although we have answered Question \ref{qn} in the special cases but the general case is still not answered. For the general case, if we try to work with Dahmani's construction, we need to identify more points in the space constructed in \cite{dahmni} but it is unclear which  points are needed to be identified.
\end{rem}
\section{Proof of Theorems \ref{rel para}, \ref{rel cyclic}}\label{6}
So far, we have proved combination theorems for convergence groups. Let $\Gamma$ be as in the Theorem \ref{para}. First of all, if each vertex group $\Gamma_v$ acts geometrically finitely on compact metrizable space $X_v$, then the group $\Gamma$ acts geometrically finitely on the space $M$ constructed in Section \ref{3} . Using this, we give the proof of Theorems \ref{rel para},\ref{rel cyclic}. To prove that $\Gamma$ acts geometrically finitely , we demonstrate that every point of $M$ is either a conical limit point or a bounded parabolic point. So, we start proving the following lemmata:
\begin{lemma}\label{conical}
	Every point in $\pt\subset M$ is a conical limit point for $\gma$ in $M$. 
\end{lemma}
\begin{proof}
	Let $\eta\in \pt \subset M$ and $v_0$ be a vertex of $T$. Then there exists a sequence $(\gamma_n)_{n\in\mathbb{N}}$ in $\gma$ such that $\gamma_nv_{0}$ lies on the unique geodesic ray $[v_0,\eta)$ for all $n$. Then by lemma \ref{large}, there exists a subsequence denoted by $\gamma_n$ and a point $p\in M$ such that for all $q$ in $M$ except possibly a point in $\pt$, we have $\gamma^{-1}_nq$ converges to $p$.
	% We apply lemma\ref{large} for the sequence $(\gamma^{-1})$ %  
	Now after multiplying each $\gamma_n$ on the the right by an element of $\Gamma_{v_0}$, we can assume that $p$ does not belong to $X_{v_0}$. To prove that $\eta$ is a conical limit point of $\Gamma$ in $M$, it is sufficient to prove that $\gamma_n\eta$ does not converges to $p$. Observe that the ray $[\gamma_n^{-1}v_{0},\gamma_n^{-1}\eta)$ always have $v_0$ on this ray for all $n$. If the sequence $\gamma_n^{-1}\eta$ converges to $p$ then as the sequence $\gamma_nx$ also converges to $p$ for any $x\in X_{v_0}$, we see that $p$ belongs to $\Gamma_{v_0}$, which is a contradiction to our choice of $p$.
\end{proof}
	
Now, we prove that each conical limit point for the action of the vertex group $\Gamma_v$ on $X_v$ is conical for the action of $\Gamma$ on $M$.

\begin{lemma}\label{conicalvertex}
	Every point in $\Omega'$ which is the image of a conical limit point in the vertex stabilizer's boundary is  a conical limit point for $\gma$ in $M$.
\end{lemma}
\begin{proof}
	Let $x\in X_v$ be a conical limit point for $\gma_v$ in $X_v$. There exists a sequence $\seq$ and two distinct points $y$ and $z$ in $X_v$ such that $\gamma_n(x) \rightarrow y$ and $\gamma_n(x')\rightarrow z$ for all $x' \neq x$. Now, we show that $\pi(x)$ is a conical limit point for $\gma$ in $M$. Since the restriction of $\pi$ to $X_v$ is continuous from $X_v$ to $M$. Therefore $\pi(\gamma_nx) = \gamma_n\pi(x) \rightarrow \pi(y)$ and $\pi(\gamma_nx') = \gamma_n\pi(x') \rightarrow \pi(z)$. Since restriction of $\pi$ is injective, so $\pi(y)$ and $\pi(z)$ are distinct. Hence $\pi(x)$ is a conical limit point for $\gma_v$ in $M$, hence for $\gma$ in $M$. 
\end{proof}
The following lemma proves that the image of each bounded parabolic point in vertex space is bounded parabolic for $\Gamma$ in $M$.
\begin{lemma}\label{bdd prbolc}
	Every point in $\Omega'$ which is image by $\pi$ of a bounded parabolic point in some vertex stabilizer's boundary is a bounded parabolic point for $\gma$ in $M$. 
\end{lemma}
\begin{proof}{\bf (Amalgam Case)}
	We  prove the lemma  in two cases:
	
	Case(1) Let $p$ be a bounded parabolic point for a vertex group $\Gamma_v$ in $X_v$, which is not in any edge space attached to $X_v$. We denote $\pi(p)$ by $p$. Let $D(p) = \{v\}$. Let $P$ be the stabilizer of $p$ in $\gma$. Since $P$ fixes the vertex $v$, $P \leq \gma_v$. In fact, $P$ is the stabilizer of $p$ in $\gma_v$. As $p$ is bounded parabolic point for $\gma_v$ in $X_v$, $P$ acts co-compactly on $X_v\setminus\{p\}$. Let $K$ be a compact subset of $X_v\setminus\{p\}$ such that $PK = X_v\setminus\{p\}$. Consider $\EE$ the set of edges whose boundaries intersect $K$. Let $e$ be the edge with one vertex $v$, then there exists $h\in P$ such that $X_e \cap hK \ne \emptyset$. Therefore the set of edges $\cup_{h\in P} h\EE $ contains every edge with one and only one vertex $v$. Let $\VV$ be the set of vertices $w$ of the tree $T$ such that the first edge of $[v,w]$ is in $\EE$. Let $K'$ be the subset of $M$ consisting of points whose domains are in $\overline{\VV}$. Define $K'' = K\cup K'$. As the sequence of points in $K'$ has limit in $K$, $K''$ is a compact subset of $M$. Now, it is clear that $PK'' = M\setminus \{p\}$.
	
	Case(2) Suppose $p$ is an element of edge boundary. In this case, $D(p)$ is infinite. Suppose that $v_1$ and $v_2$ are vertices of an edge $e$. Let $X_e = \{p\}$ and $P_1,P_2$ be maximal parabolic subgroups in vertex groups $\Gamma_{v_1},\Gamma_{v_2}$ respectively, and $P$ is the parabolic edge group. Then $D(p)$ is nothing but the Bass-Serre tree of the amalgam $Q=P_1\ast_P P_2$, which is the stabilizer of $p$ in $\Gamma$. Under the action of $P_1\ast_P P_2$ the quotient of $D(p)$ is the edge $e$. Since $P_1,P_2$ acts co-compactly on $X_{v_1}\setminus\{p\},X_{v_2}\setminus\{p\}$ respectively, there exists a compact subset $K_i$ of $X_{v_i}\setminus\{p\}$  such that $P_iK_{i} = X_{v_i}$ for $i=1,2$. Consider $\EE_i$ the set of edges starting at $v_i$ whose boundary intersects $K_i$ but does not contain $p$. Let $e$ be an edge with only one vertex in $D(p)$ and $v_i$ be this vertex. Then there exists $h\in P_i$ such that $X_e\cap hK_i \ne \emptyset$ for $i=1,2$. Therefore the set of edges $\cup_{i=1,2} Q\EE_i$ contains every edge with one and only one vertex in $D(p)$. Let $\VV_i$ be the set of vertices $w$ such that first edge of $[v_i,w]$ is in $\EE_i$. Let $K_i'$ be the subset of $M$ consisting of the points whose domain is in $\VV_i$ for $i=1,2$. Since a sequence of points in the spaces corresponding to the stabilizers of distinct edges in $\EE_i$ have only accumulation points in $K_i$, the set $K_i''= K_i\cup K_i'$ is a compact for $i=1,2$. Hence $K=\cup_{i=1,2}K_i''$ is a compact set of $M$ not containing $p$ and $QK =M\setminus\{p\}$. Therefore $p$ is bounded parabolic point for $\Gamma$ in $M$.
	
 {\bf (HNN extension case)} Case(1) Let $\Gamma_v$ be vertex group in HNN extension and $P$ is a parabolic subgroup sitting inside a maximal parabolic subgroup $P_1$  and isomorphic to a subgroup $P'$ of $P_1$.  In this case, the proof remains the same as in the amalgam case except that the maximal parabolic subgroup corresponding to edge boundary point is $P_1\ast_{P\simeq P'}$, and maximal parabolic subgroups corresponding to parabolic points which are not in any edge spaces are maximal parabolic for $\Gamma$ in $M$.

Case(2) Let $\Gamma_v$ be same as in case(1) and suppose $P$ is sitting inside $P_1$ and is isomorphic to a subgroup $P'$ of maximal parabolic subgroup $P_2$, which is not conjugate to $P_1$ in $\Gamma_v$. Then, in this case, we can write $\Gamma_v\ast_{P\simeq P'} = (\Gamma_v \ast_P P')\ast_{P'}$, and we apply the amalgam and case(1) of HNN extension respectively to get the result.
\end{proof}
From the above lemmata, it is clear that if each $\Gamma_v$ acts on $X_v$ geometrically finitely then $\Gamma$ acts on $M$ geometrically finitely. Now, we are in the position of proving the following proposition: 
\begin{prop}
	Let $\Gamma$ be a finitely generated group that splits as a finite graph of convergence groups with finitely generated parabolic edge groups. Then $\Gamma$ acts geometrically finitely on $M$ (constructed in Section \ref{3}) if and only if each vertex group $\Gamma_v$ acts geometrically finitely on compact metrizable space $X_v$.
\end{prop} 
\begin{proof}
	Note that, by \cite[Lemma 2.5]{bigdely}, each vertex group $\Gamma_v$ is finitely generated. Suppose each vertex group $\Gamma_v$ acts geometrically finitely on $X_v$. Then, by Lemma \ref{conical}, \ref{conicalvertex}, \ref{bdd prbolc}, $\Gamma$ acts geometrically finitely on $M$. Conversely, suppose that $\Gamma$ acts geometrically finitely on $M$. Since each edge group is parabolic for the action of $\Gamma$ on $M$, each edge group is a relatively quasi-convex subgroup of $\Gamma$. By \cite[Proposition 5.2]{haulmark2021}, each vertex group $\Gamma_v$ is a relatively quasi-convex subgroup of $\Gamma$. In particular, each $\Gamma_v$ acts geometrically finitely on $X_v$ as $X_v$ is the limit set of $\Gamma_v$ for $\Gamma$ acting on $M$.
\end{proof}

\subsection*{Proof of \thmref{rel para}} Let $\Gamma$ be either amalgam or HNN extension of relatively hyperbolic groups with parabolic edge groups. Since each vertex group $\Gamma_v$ is relatively hyperbolic, $\Gamma$ acts geometrically finitely on its Bowditch boundary. To prove $\Gamma$ is relatively hyperbolic, we use Yaman's characterization Theorem \ref{thm1}. For constructing a space on which $\Gamma$ acts geometrically finitely, we follow the same construction as in Section \ref{3} by taking compactum $X_v$, $X_e$ for $\Gamma_v$, $\Gamma_e$ as Bowditch boundaries of these groups, respectively. Let $\Gamma$ be as in the theorem. From the above discussion, we have a space $M$ which is compact metrizable as proved in Section \ref{3}. Now, the proof of Theorem \ref{para} gives that $\Gamma$ acts on $M$ as a convergence group. Since each vertex group acts geometrically on its Bowditch boundary, from the above lemmata, the groups $\Gamma$ acts geometrically finitely on $M$. Hence by Yaman's characterization \ref{thm1}, $\Gamma$ is relatively hyperbolic, and $M$ is Bowditch boundary for $\Gamma$. \qed

Let $\Gamma$ be as in the proof of the above theorem. Consider the collection $\mathcal{G}$ containing the two type of subgroups of $\Gamma$: (1) stabilizers of bounded parabolic points in Bowditch boundary of vertex groups which are not identified with edge parabolic point. (2) stabilizers of edge parabolic points in $\Gamma$. Then $\Gamma$ is hyperbolic relative to $\mathcal{G}$.
\begin{rem}
	 The limit set of each vertex group, $\Gamma_v$ for the action of $\Gamma$ on $M$, is homeomorphic to its Bowditch boundary $\partial\Gamma_v$. It is clear that $\Gamma_v$ acts geometrically finitely on its limit set, and therefore, $\Gamma_v$ is a relatively quasi-convex subgroup of $\Gamma$.
\end{rem}
 Now, we prove Theorem \ref{rel cyclic}. For that, let us recall some basic definitions and results from \cite{osinelem}.
 
  Let $G$ be a group hyperbolic relative to a collection of subgroups $\{H_{\alpha},\alpha\in \Lambda\}$. A subgroup $Q$ of $G$ is said to be \emph{hyperbolically embedded} in $G$ if $G$ is hyperbolic relative to $\{H_{\alpha},\alpha\in \Lambda\}\cup \{Q\}$. We say that an element $g$ of $G$ is \emph{parabolic} if it is conjugate to an element of $H_\alpha$ for some $\alpha\in \Lambda$. Otherwise, it is called \emph{hyperbolic}. For any hyperbolic element $g$ of infinite order, we set $E(g)=\{f\in G : f^{-1}g^nf=g^{\pm}n\}$. For any hyperbolic element of infinite order, Osin proved the following theorem:
\begin{theorem}\label{osin}\cite[Theorem 4.3]{osinelem}
	Every hyperbolic element $g$ of infinite order in $G$ is contained in a unique maximal elementary subgroup $E(g)$.
\end{theorem}
From proof of the above theorem, it follows that $[E(g):\langle g\rangle ]<\infty$ and therefore $E(g)$ is elementary. Also, for infinite order hyperbolic element, the unique maximal elementary subgroup is hyperbolically embedded in the relatively hyperbolic group $G$, see \cite[Corollary 1.7]{osinelem}. To prove Theorem \ref{rel cyclic}, as we mentioned in the introduction, it is sufficient to consider the amalgam and the HNN extension case. 

\subsection*{Proof of \thmref{rel cyclic}}
Case(1): Let $\Gamma = \Gamma_1\ast_{Z_1\simeq Z_2}\Gamma_2$. Suppose $Z_1=\langle\gamma_1\rangle$ and $Z=\langle \gamma_2\rangle$. If both $\gamma_1,\gamma_2$ are parabolic elements in $\Gamma_1,\Gamma_2$ respectively, then we are in amalgam case of Theorem \ref{rel para}, and hence $\Gamma$ is relatively hyperbolic. Suppose at least one of them is a hyperbolic element, then by Theorem \ref{osin}, we have a maximal elementary subgroup containing cyclic subgroup generated by a hyperbolic element which is hyperbolically embedded. Thus, we are in the amalgam case of the Theorem \ref{rel para} and therefore $\Gamma$ is relatively hyperbolic.\\
Case(2): Let $\Gamma=\Gamma_1\ast_{Z\simeq Z'}$ and let $Z=\langle \gamma\rangle, Z'=\langle \gamma'\rangle$ are isomorphic subgroups of $\Gamma_1$. Again if both $\gamma,\gamma'$ are parabolic elements, then we are in the HNN extension case of Theorem \ref{rel para}, and hence $\Gamma$ is relatively hyperbolic. If at least one of them is a hyperbolic element, then by applying the Theorem \ref{osin}, we get maximal elementary subgroups containing that cyclic subgroup that is hyperbolically embedded. Thus, we are in the HNN extension case of Theorem \ref{rel para}, and hence $\Gamma$ is relatively hyperbolic.

In either case, we have proved that $\Gamma$ is relatively hyperbolic and applying Theorem \ref{rel para}, we also have a description of Bowditch boundary. \qed

\section{Homeomorphism type of Bowditch boundary}
In \cite{alex}, authors proved that the homeomorphism type of Gromov boundary of the fundamental group of a graph of hyperbolic groups with finite edge groups depends only on the set of homeomorphism type of Gromov boundary of non-elementary hyperbolic vertex groups. It is not clear that the same result can be extended in the case of a graph of relatively hyperbolic groups with finite edge groups. However, under some assumptions, we prove a similar result for a graph of relatively hyperbolic groups with parabolic edge groups. For the convenience of the reader, we are again stating the following theorem:
\begin{theorem}\label{homeotype}
	Let $Y$ be a finite connected graph and let $G(Y),G'(Y)$ be two graph of groups satisfying the following:
	\begin{enumerate}
		\item For each vertex $v\in V(Y)$, let $(G_v,\mathbb{P}_v),(G_v',\mathbb{P}_v')$ be relatively hyperbolic groups.
		\item Let $e\in E(Y)$ be an edge with vertices $v,w$ and let $P_e,P_e'$ are parabolic edge groups in $G(Y),G'(Y)$, respectively. Then either $P_e,P_e'$ have infinite index in corresponding maximal parabolic subgroups in $G_v,G_v'$, respectively or $P_e,P_e'$ have the same finite index in maximal parabolic subgroups in $G_v,G_v'$, respectively. Similarly, either $P_e,P_e'$ have infinite index in maximal parabolic subgroups in $G_w,G_w'$, respectively or $P_e,P_e'$ have the same finite index in corresponding maximal parabolic subgroups in $G_w,G_w'$, respectively.
		\item Let $B_v$, $B_v'$ be the set of translates of parabolic points corresponding to adjacent edge groups under the action of $G_v,G_v'$ on their Bowditch boundaries respectively. For each vertex $v\in V(Y)$, suppose we have a homeomorphism from $\partial G_v\rightarrow \partial G_v'$ that maps $B_v$ onto $B_v'$.
	\end{enumerate}
Let $\Gamma=\pi_1(G(Y))$, $\Gamma'=\pi_1(G'(Y))$. (By Theorem \ref{rel para}, the groups $(\Gamma,\mathbb{P}),(\Gamma',\mathbb{P'})$ are relatively hyperbolic) Then there exists a homeomorphism from $\partial \Gamma$ to $\partial \Gamma'$ preserving edge parabolic points, i.e. taking parabolic points corresponding to edge groups of $G(Y)$ to parabolic points corresponding to edge groups of $G'(Y)$.
\end{theorem}
\begin{rem}
	In the above theorem, it is not possible that for $G(Y)$, edge groups are maximal parabolic, and for $G'(Y)$, edge groups are parabolic (Not maximal parabolic). Thus in both graphs of groups, either edge groups are maximal parabolic or edge groups are parabolic. The example below justifies this situation.
\end{rem}
\begin{example}
	Let $Y$ be an edge and let $F(a,b)$ be a free group of rank $2$. Let $G(Y)$ be a double of free group $F(a,b)$ along $\langle[a,b]\rangle$ and $G'(Y)$ be a double of free group $F(a,b)$ along $\langle[a,b]^2\rangle$. Thus $\Gamma= F(a,b)\ast_{\langle[a,b]\rangle \simeq \langle[\bar{a},\bar{b}]\rangle} F(\bar{a},\bar{b})$ and $\Gamma'= F(a,b)\ast_{\langle[a,b]^2\rangle \simeq \langle[\bar{a},\bar{b}]^2\rangle} F(\bar{a},\bar{b})$. Assume that $F(a,b)$ is relatively hyperbolic with respect to $\langle[a,b]\rangle$. By Theorem \ref{rel para}, $\Gamma$ and $\Gamma'$ are relatively hyperbolic with respect to $\langle[a,b]\rangle$ and $\langle[a,b]\rangle\ast_{\langle[a,b]^2\rangle \simeq \langle[\bar{a},\bar{b}]^2\rangle} \langle[\bar{a},\bar{b}]\rangle$ respectively. Here, we just take identity map between Bowditch boundaries of vertex groups. From the construction of Bowditch boundaries, it is clear that if we remove parabolic points from Bowditch boundaries $\partial \Gamma, \partial \Gamma'$ respectively, we have two connected components, infinitely many connected components, respectively. Therefore there is no homeomorphism from $\partial \Gamma$ to $\partial \Gamma'$ preserving edge parabolic points.
\end{example}
Also, the above theorem does not deal with the case when parabolic edge groups have different finite indexes in corresponding maximal parabolic subgroups. Here, we give a specific example in this direction.
\begin{example}
	Consider the two groups $\Gamma= F(a,b)\ast_{\langle[a,b]\rangle \simeq \langle[\bar{a},\bar{b}]\rangle } F(\bar{a},\bar{b})$ and $\Gamma' = F(a,b)\ast_{\langle[a,b]\rangle \simeq \langle[a,b]^2\rangle} F(\bar{a},\bar{b})$. Assume that $F(a,b)$ is relatively hyperbolic with respect to $\langle[a,b]\rangle$. Both the groups $\Gamma,\Gamma'$ are relatively hyperbolic with Bowditch boundary $\partial \Gamma$, $\partial \Gamma'$ respectively. Again, from the construction of Bowditch boundary, removing a parabolic point from $\partial \Gamma$ gives two connected components but removing a parabolic point from $\partial \Gamma'$ gives three connected components. Thus there is no homeomorphism from $\partial \Gamma$ to $\partial \Gamma'$ preserving edge parabolic points. Similarly, if we take $\Gamma= F(a,b)\ast_{\langle[a,b]\rangle\simeq \langle[a,b]^3\rangle } F(\bar{a},\bar{b})$ and $\Gamma' = F(a,b)\ast_{\langle[a,b]\rangle \simeq \langle[a,b]^2\rangle} F(\bar{a},\bar{b})$ then there is no homeomorphism from $\partial \Gamma$ to $\partial \Gamma'$ preserving edge parabolic points.
\end{example}
To prove Theorem \ref{homeotype}, it is sufficient to consider amalgam and HNN extension case.
%%%%%%%%%%%%%%%%%%%%%%%%%%%%%%%%%%%%%%%%%%%%%%%%%%%%%%%%%%%%%%%
\subsection{ Proof of Theorem \ref{homeotype} in amalgam case:} 

Let $Y$ be an edge. Then $G(Y)$ and $G'(Y)$ are amalgams of two relatively hyperbolic groups with parabolic edge groups. Let $\Gamma=\pi_1(G(Y))$ and let $\Gamma'=\pi_1(G'(Y))$. Let $T,T'$ be the Bass-Serre trees for $G(Y),G'(Y)$, respectively. For each edge $e\in T$ and $e'\in T'$, let $P_e,P_e'$ be parabolic edge groups in $\Gamma,\Gamma'$ respectively. Also, in adjacent vertices of $e$ and $e'$, let $P_v,P_w$ and let $P_v',P_w'$ be maximal parabolic subgroups corresponding to $P_e,P_e'$ respectively. Let $\partial \Gamma,\partial \Gamma'$ denotes Bowditch boundaries of $\Gamma,\Gamma'$ respectively. Keeping the construction of Bowditch boundaries in mind, we define a map $f$ from $\partial \Gamma$ to $\partial \Gamma'$. 

Let $e,e'$ be two edges of $T,T'$ with vertices $v,w$ and $v',w'$ respectively. Suppose we have homeomorphisms $\partial G_v\rightarrow \partial G_v'$ and $\partial G_w\rightarrow \partial G_w'$ as in Theorem \ref{homeotype}(3). By definition of these homeomorphisms, we have bijections between cosets of $P_v$ in $G_v$ and cosets of $P_v'$ in $G_v'$. Also, there is a bijection between cosets of $P_e$ in $P_v$ and cosets of $P_e'$ in $P_v'$. Combining these two, we get a bijection between cosets of $P_e$ in $G_v$ and cosets of $P_e'$ in $G_v'$. Similarly, we have a bijection between cosets of $P_e$ in $G_w$ and cosets of $P_e'$ in $G_w'$. By following this process inductively, we have an isomorphism $\phi$ from $T$ to $T'$. Let $\xi\in \partial G_v$ for some vertex $v\in V(T)$. Define $f(\xi):=f_v(\xi)$, where $f_v$ is a homeomorphism from $\partial G_v$ to $\partial G_{\phi(v)}$. Note that if $\xi$ is a parabolic point in $\partial G_v$ and let $D(\xi)$ be its domain then $\phi|_{D(\xi)}=D(f(\xi))$. Since $\phi$ is an isomorphism, we have a homeomorphism $\partial \phi$ from $\partial T$ to $\partial T'$. Observe that if some point of $\partial T$ is identified with some parabolic point, then its image under $\partial \phi$ is also identified with some parabolic point. If $\eta \in \partial T$ such that it is not identified with some edge parabolic point, define $f(\eta):= \phi(\eta)$. Clearly, $f$ is a bijection. Thus, we have a map $f$ from $\partial \Gamma$ to $\partial \Gamma'$. To prove that $f$ is a homeomorphism, it is sufficient to prove that $f$ is continuous as Bowditch boundaries $\partial \Gamma,\partial \Gamma'$ are compact Hausdorff. Let $\xi\in \partial G_v$ for some $v\in V(T)$ and let $U$ be a neighborhood of $f(\xi)$ in $\partial \Gamma'$. Note that $\phi(D(\xi))=D(f(\xi))$. For each vertex $u\in D(\xi)$, we can choose a neighborhood $V_u$ such that $f_u(V_u)\subset U_{\phi(u)}$ as $f_u$ is a homeomorphism, where $U_{\phi(u)}$ is a neighborhood around $f(\xi)$ in $\partial G_{\phi(u)}$. Now, it is clear from the definition of neighborhoods in $\partial \Gamma$ and definition of maps $f,\phi$ that we can find a neighborhood $V$ of $\xi$ in $\partial \Gamma$ such that $f(V)\subset U$. Now, let $\eta \in \partial T$ such that it is not identified with a parabolic point. Let $U$ be a neighborhood of $f(\eta)$ in $\partial \Gamma'$.
From the construction of map $\phi$, it is clear that $\phi$ takes subtree $W_m(\eta)$ (see section \ref{3} ) onto the subtree $W_m(f(\eta))$. Then again, by definition of neighborhoods, we can find a neighborhood $V$ of $\eta$ in $\partial \Gamma$ such that $f(V)\subset U$. The map $f$ is continuous, and hence $f$ is a homeomorphism. \qed
%%%%%%%%%%%%%%%%%%%%%%%%%%%%%%%%%%%%%%%%%%%%%%%%%%%%%%%%%%%%%%%
\subsection{Proof of Theorem \ref{homeotype} in HNN extension case:} Here, we prove it in the following two subcases:

Subcase(1) Let $\Gamma=G\ast_{P_1\simeq P_2},\Gamma'=G'\ast_{P_1'\simeq P_2'}$, where $(G,\mathbb{P}),(G',\mathbb{P'})$ are relatively hyperbolic groups. Assume that $P_1,P_2$, and $P_1',P_2'$ are both sitting inside the same maximal parabolic subgroups in $G,G'$, respectively. Let $T,T'$ be the Bass-Serre trees of $\Gamma,\Gamma'$ respectively. Also, we have a homeomorphism between Bowditch boundaries $\partial G,\partial G'$ satisfying (3) in Theorem \ref{homeotype}. We get an isomorphism $\phi$ from $T$ to $T'$ in a similar manner as in the amalgam case. Also, we can define a map from $\partial \Gamma$ to $\partial \Gamma'$ in the same way as we define in the amalgam case. Note that $f$ is a bijection. To prove that $f$ is a homeomorphism, it is sufficient to prove that $f$ is continuous as $\partial \Gamma,\partial \Gamma'$ are compact Hausdorff. Again continuity is clear from the definition of map $f$ and definition of neighborhoods in $\partial \Gamma, \partial \Gamma'$ respectively.

Subcase(2) Let $\Gamma=G\ast_{P_1\simeq P_2},\Gamma'=G'\ast_{P_1'\simeq P_2'}$, where $(G,\mathbb{P}),(G',\mathbb{P'})$ are relatively hyperbolic groups. In this case, $P_1,P_2$, and $P_1',P_2'$ are sitting inside in different(not conjugate) maximal parabolic subgroups in $G,G'$, respectively. Now, we can write $\Gamma =(G\ast_{P_1}P_2)\ast_{P_2}$ and $\Gamma'=(G'\ast_{P_1'}P_2')\ast_{P_2'}$, respectively. By applying amalgam and Subcase(1) of the HNN extension respectively, we get the desired homeomorphism from $\partial \Gamma$ to $\partial \Gamma'$. \qed
%%%%%%%%%%%%%%%%%%%%%%%%%%%%%%%%%%%%%%%%%%%%%%%%%%%%%%%%%%%%%%%%%%%%%%%%%%%%%%%%%%%%%%%%%%%%%%%%%%%%%%%%%%%%%%%%%%%%%%%%%%%%%
\section{Applications and examples}
\subsection{Example of a subgroup of a relatively hyperbolic group with exotic limit set}

In this subsection, following the construction of space given in Section \ref{3}, we give an example of a relatively hyperbolic group having a non-relatively quasi-convex subgroup whose limit set is not equal to the limit of any  relatively quasi-convex subgroup. This is motivated by the work of I.Kapovich \cite{ilya-kapovich}, where he gave such an example in the case of hyperbolic group.
 
 Consider the torus with one puncture and let $\psi$ be a pseudo-Anosov homeomorphism fixing the puncture. Suppose $M_{\psi}$ is the mapping torus for the homeomorphism $\psi$. Let $G,F$ be the fundamental groups of $M_{\psi}$, puncture torus respectively. Then it is well known that $G$ is relatively hyperbolic with respect to a subgroup isomorphic to $Z\oplus Z$, and subgroup $F$ is not relatively quasi-convex in $G$. Let $\Gamma=G\ast_{\langle z\rangle}\overline{G}$ and let $H=F\ast_{\langle z\rangle}\overline{F}$, where $z\in F$, be doubles of $G$ and $H$ respectively along cyclic subgroup $\langle z\rangle$. The groups $\Gamma,H$ are relatively hyperbolic, by Theorem \ref{rel cyclic}. We have the following:
 \begin{lemma}
 	$H$ is not a relatively quasi-convex subgroup of $\Gamma$.
 \end{lemma}
 \begin{proof}
 	Suppose $H$ is relatively quasi-convex in $\Gamma$. Since $F$ is relatively quasi-convex in $H$ and $H$ is relatively quasi-convex in $\Gamma$, $F$ is relatively quasi-convex in $\Gamma$ by \cite[Lemma 2.3]{bigdely}. Since $F$ is a normal subgroup of $G$, the limit set of $F$ in $G$ is the same as Bowditch boundary of $G$ that is homeomorphic to the limit set of $G$ in $\Gamma$. Also the limit set of $F$ in $\Gamma$ is the same as Bowditch boundary of $G$. Hence $F$ acts geometrically finitely on its limit in $G$. Thus, $F$ is relatively quasi-convex in $G$, which is a contradiction.
 \end{proof}
As we observe in the proof of Theorem \ref{rel cyclic}, the edge group in $\Gamma$ is parabolic, or we change the parabolic structure in vertex group $G$ so that edge group become parabolic. Therefore, following Section \ref{3}, we can give a construction of Bowditch boundary of relatively hyperbolic group $\Gamma$. Now, we prove the following:
\begin{lemma}\label{maximal}
	$Stab_{\Gamma}(\Lambda(H))=H$, i.e. $H$ is maximal in its limit set.
\end{lemma}
\begin{proof}
	Let $T,T'$ be the Bass-Serre trees of the groups $\Gamma,H$ respectively. Note that the tree $T'$ embeds in $T$, and each vertex group in $H$ is normal in the corresponding vertex group of $\Gamma$. Also, there is topological embedding between Gromov boundaries of $T'$ and $T$. Let $M$ be Bowditch boundary of $\Gamma$. Here we explicitly know the construction of $M$ (see Section \ref{3}). Let $p$ be a map from $M\rightarrow T\cup \partial T$ defined as follows: for $\xi\in \partial G_v$, define $p(\xi)=v$ and for $\eta\in \partial T$, define $p(\eta)=\eta$. Consider $N$, a subset of $M$, the inverse image of $T'\cup \partial T'$ under the map $p$. It is clear from the definition of topology on $M$ that $N$ is the minimal closed $H$-invariant set. Thus $\Lambda(H)=N$. Now, the lemma follows immediately from the construction of Bowditch boundary $M$.
\end{proof}
Now, we prove that the limit set of the subgroup $H$ is exotic, i.e. there is no relatively quasi-convex subgroup of $\Gamma$ whose limit set is equal to the limit set of $H$. Recall that a subgroup of a relatively hyperbolic group is dynamically quasi-convex if and only if it is relatively quasi-convex (see \cite{gerasomov-potya}).
\begin{lemma}
	There does not exist a relatively quasi-convex subgroup of $\Gamma$ whose limit set is equal to the limit set of $H$. 
\end{lemma}
\begin{proof}
	If possible, there is a relatively quasi-convex subgroup $Q$ of $\Gamma$ such that $\Lambda(Q)=\Lambda(H)$. Since $Stab_{\Gamma}(\Lambda(H))=Stab_{\Gamma}(\Lambda(Q))=H$, we see that $Q\subset H$. Thus, by \cite[Lemma 2.6]{yang}, we see that $H$ is a dynamically quasi-convex subgroup of $\Gamma$; it is relatively quasi-convex, which gives a contradiction as $H$ is not relatively quasi-convex.
\end{proof}
\begin{rem}
	In \cite{dahmni}, Dahmani gave a construction of Bowditch boundary for the fundamental group of an acylindrical graph of relatively hyperbolic groups with fully quasi-convex edge groups. In particular, we can construct Gromov boundary of the fundamental group of an acylindrical graph of hyperbolic groups with quasi-convex edge groups. Let $S_g,g\geq 2$ be a closed orientable surface of genus $g$ and let $\phi$ be a pseudo-Anosov homeomorphism of $S_g$. Let $M_{\phi}$ be the mapping torus corresponding to $\phi$ and let $G$ be the fundamental group of $M_{\phi}$. Then $G$ is a hyperbolic group and $F=\pi_1(S_g)$ is a non-quasiconvex subgroup of $G$. Let $z\in F$ be such that $z$ is not a proper power in $F$ and hence it is not a proper power in $G$. Consider the double $\Gamma$ of group $G$ along $\langle z\rangle$, i.e. $\Gamma=G\ast_{\langle z\rangle}\overline{G}$. Note that the group $\Gamma$ is hyperbolic by \cite{bestcom} and the subgroups $G, \overline{G}$ of $\Gamma$ are quasi-convex. Consider the group $H=F\ast_{\langle z\rangle}\overline{F}$. Again, by \cite{bestcom}, $H$ is hyperbolic. Now, using the construction of Gromov boundary of $\Gamma$ from \cite{dahmni}, we can explicitly construct the limit set of subgroup $H$ (as we did in Lemma \ref{maximal}). Then, we have $Stab_{\Gamma}\Lambda(H)=H$. Then, using the same idea as above, $H$ is not quasi-convex, and there is no quasi-convex subgroup of $\Gamma$ whose limit set equal $\Lambda(H)$. Thus, we have a different proof of I.Kapovich's result from \cite{ilya}.
\end{rem}
Above, we have given an example of a graph of relatively hyperbolic groups for which we have a subgraph of groups such that the fundamental group of subgraph of groups is not relatively quasi-convex in the fundamental group of the graph of groups. Contrary to that, we prove that when the subgraph of groups is obtained by restricting graph of groups (as in Theorem \ref{rel para}) to a subgraph, then the fundamental group of the subgraph of groups is relatively quasi-convex in the fundamental group of graph of groups.

Let $\mathcal{G}(\mathcal{Y})$ be a graph of groups as in Theorem \ref{rel para} and $\mathcal{G}(\mathcal{Y}_1)$ be a subgraph of groups obtained by restricting graph of groups $\mathcal{G}(\mathcal{Y})$ to a subgraph $Y_1\subset Y$. Let $\Gamma,\Gamma_1$ be the fundamental groups of $\mathcal{G}(\mathcal{Y})$, $\mathcal{G}(\mathcal{Y}_1)$, respectively. Let $T_1$ be the Bass-Serre tree for $\mathcal{G}(\mathcal{Y}_1)$. As an application of Theorem \ref{rel para}, we have the following:
\begin{prop}
	$\Gamma_1$ is a relatively quasi-convex subgroup of $\Gamma$.
\end{prop}
\begin{proof}
	$\mathcal{G}(\mathcal{Y}_1)$ is also a graph of relatively hyperbolic groups with parabolic edge groups. Using the construction of Section \ref{3}, we have a set $M_1\subset M$ for $\Gamma_1$. By definition of neighborhoods on $M$, we see that $M_1$ is a closed subset of $M$.  It is also minimal $\Gamma_1$-invariant subset of $M$. Thus, $M_1$ is the limit set of $\Gamma_1$. By the proof of Theorem \ref{rel para}, it is clear that every point of $M_1$ coming from vertex boundaries is either conical or bounded parabolic. Now, let $\eta$ be a point in $(\partial T_1)'$. Since the action of $\Gamma_1$ on $T_1$ is co-compact, there exists a sequence $(\gamma_{1n})_{n\in \mathbb{N}}\subset \Gamma_1$ and a vertex $v\in T_1$ such that $\gamma_{1n}v$ meet the geodesic ray $[v_0,\eta)$ for each $n$. Now by applying the proof of Lemma \ref{conical}, we see that up to extraction, a subsequence of $(\gamma_{1n}^{-1})$, $(\gamma_{1n}^{-1})$ is a conical sequence for the action of $\Gamma$ on $M$. Since $M_1$ is a closed subset of $M$ and $(\gamma_{1n})_{n\in \mathbb{N}}\subset \Gamma_1$, $(\gamma_{1n}^{-1})$ is conical for the action of $\Gamma_1$ on $M_1$. Thus the proposition follows.
\end{proof}
\subsection{Example of a family of non-convergence groups}
Here, we give an example of a family of groups that do not act on a compact metrizable space as a non-elementary convergence group. 
\begin{prop}\label{notconv}
	Let $G$ be a torsion-free group and let $H$ be a subgroup of $G$ satisfying the following:
\begin{itemize}
	\item $H$ is malnormally closed in $G$, i.e. there is no proper subgroup of $G$ containing $H$, which is malnormal in $G$
	\item $[Comm_G(H):H]>1$, where $Comm_G(H)$ denotes the commensurator of $H$ in $G$.
\end{itemize}
Consider the double $\Gamma$ of the group $G$ along $H$, i.e. $\Gamma=G\ast_{H\simeq \overline{H}}\overline{G}$. Then, $\Gamma$ does not act on a compact metrizable space as a non-elementary convergence group.
\end{prop}
 First of all, we collect some basic facts about subgroups of a convergence group.
\begin{lemma}
	Suppose $G$ acts on a compact metrizable space as a convergence group. Then a subgroup $P$ of $G$ isomorphic to $Z\oplus Z$ is parabolic.
\end{lemma}
\begin{proof}
	 Since $P$ is abelian, $|\Lambda(P)|\leq2$. Let if possible $|\Lambda(P)|=2$. Then $P$ contains a loxodromic element $p$ (say). The fixed point set of $p$, Fix$(p)=\Lambda(P)$ and $Stab_G(\Lambda(P))$ contains $P$. By \cite[Theorem 2I]{tukiametric}, $\langle p\rangle$ has finite index in $Stab_G(\Lambda(P))$. In particular, $\langle p\rangle$ has finite index in $P$, which is impossible. Hence $|\Lambda(P)|=1$ and $P$ is a parabolic subgroup.
\end{proof}
A subgroup $K$ of $G$ is said to be weakly malnormal if for all $g\in G\setminus H$, $|H\cap gHg^{-1}|<\infty$. We  observe the following:
\begin{lemma}\label{malnormal}
	Let $G$ be a group that acts on a compact metrizable space as a convergence group. Then maximal parabolic subgroups are weakly malnormal.
\end{lemma}
\begin{proof}
	Let $P$ be a maximal parabolic subgroup with $\Lambda(P)=\{p\}$. Suppose for some $g\in G$, $|P\cap gPg^{-1}|=\infty$. As $P\cap gPg^{-1}\subset P$, $\Lambda(P\cap gPg^{-1})=\Lambda(P)=\{p\}$. Similarly, $\Lambda(P\cap gPg^{-1})=\Lambda(gPg^{-1})=g\Lambda(P)$. Hence $g$ fixes the parabolic point $p$ and therefore $g\in P$.
\end{proof}
Note that if $G$ is torsion-free in the above lemma, then maximal parabolic subgroups are malnormal.
Now, we obtain the following:
\begin{lemma}\label{double para}
	Let $\Gamma$ be the double of $G$ along $H$ as above. If $H\cap gHg^{-1}\neq \{1\}$ then $H\cap gHg^{-1}$ is a subgroup of a parabolic subgroup. 
\end{lemma}
\begin{proof}
	Let $w\in H\cap gHg^{-1}$. Then $w=ghg^{-1}$ for some $h\in H$. This implies $h=g^{-1}wg$. Since $\Gamma$ is double of group $G$ along $H$, $g^{-1}wg=\bar{g}^{-1}w\bar{g}$. Thus $\bar{g}g^{-1}$ commute with $w$. As $\Gamma$ is torsion-free, $\langle \bar{g}g^{-1},w\rangle\simeq Z\oplus Z$. By the above lemma, $\langle \bar{g}g^{-1},w\rangle$ is a parabolic subgroup. Since $w$ is an arbitrary element of $H\cap gHg^{-1}$, $H\cap gHg^{-1}$ is a subgroup of a parabolic subgroup. 
\end{proof}
{\bf Proof of Propostion \ref{notconv}:} Since $[Comm_G(H):H]>1$, let $g\in Comm_G(H)$ such that $g\notin H$. By Lemma \ref{double para}, $H\cap gHg^{-1}$ is a subgroup of a parabolic subgroup of $\Gamma$. Since $[H:H\cap gHg^{-1}]<\infty$, $H$ and $H\cap gHg^{-1}$ sit inside the same maximal parabolic subgroup $P$(say). Now, by Lemma \ref{malnormal}, the subgroup $P$ is malnormal in $\Gamma$. As $H$ is malnormally closed in $G$, $G\subset P$. Similarly, we can show that $\overline{G}\subset P$. Hence $\Gamma \subset P$ and therefore $\Gamma$ is not a non-elementary convergence group.                             \qed

{\bf Note:} Let $H,K$ be two subgroups of $G$ such that $H\subset K\subset G$ and $1<[K:H]<\infty$. Then one can check that $K\subset{\tiny } Comm_G(H)$. Also, assume that $H$ is malnormally closed in $G$. Then the double $\Gamma=G\ast_{H\simeq \overline{H}}\overline{G}$ is not a non-elementary convergence group.

By Floyd mapping theorem in \cite{gerasimovfloyd}, we see that relatively hyperbolic groups have non-trivial Floyd boundary. The converse of this fact is a question of Olshanskii-Osin-Sapir \cite[Problem 7.11]{olshanskiiosinsapir}, i.e. if a finitely generated group has non-trivial Floyd boundary then it is hyperbolic relative to a collection of proper subgroups. We have the following question:
\begin{qn}\label{qn2}
	Give an example of a non-elementary convergence group that is not relatively hyperbolic?
\end{qn} 
Answer to the above question does not give a counterexample to Olshanskii-Osin-Sapir conjecture because of the following:
\begin{qn}\label{qn3}
Give an example of a non-elementary convergence group with a trivial Floyd boundary?
\end{qn}
Answer to Question \ref{qn3} will also answer Question \ref{qn2}. Given a finitely presented group $Q$, Rips has constructed a hyperbolic group $G$ and a surjection from $G$ to $Q$. Let $N$ be the kernel of this surjection. Then it is not known whether $N$ is relatively hyperbolic with respect to a collection of proper subgroups of $N$ or if it has trivial Floyd boundary.
%%%%%%%%%%%%%%%%%%%%%%%%%%%%%%%%%%%%%%%%%%%%%%%%%%%%%%%%%%%%%%%%%%%%%%%%%%%%%%%%%%%%%%%%%%%%%%%%%%%%%%%%%%%%%%%%%%%%%%%%%%%%%%

\end{document}